\documentclass[a4paper,11pt, twoside]{article}

\usepackage{amsmath}
\usepackage{geometry}
\usepackage{amssymb}
\usepackage{amsthm}
\usepackage{graphicx}
\usepackage{mathtools}
\newcommand{\trace}{\operatorname{tr}}
\newcommand{\vect}{\operatorname{vec}}
\newcommand{\diag}{\operatorname{diag}}
\newcommand{\cl}{\operatorname{cl}}
\newcommand{\blkdiag}{\operatorname{blkdiag}}
\newcommand{\expn}{\operatorname{e}}
\newcommand {\smat}      [1] {\left[\begin{smallmatrix}{#1}}
\newcommand {\srix}          {\end{smallmatrix}\right]}
\newcommand{\sig}{\operatorname{sig}}
\newcommand{\mor}{\operatorname{MOR}}
\newcommand{\reduced}{\operatorname{red}}

\usepackage{xcolor}%nur für Kommentar
\DeclareFontFamily{U}{mathx}{}
\DeclareFontShape{U}{mathx}{m}{n}{<-> mathx10}{}
\DeclareSymbolFont{mathx}{U}{mathx}{m}{n}
\DeclareMathAccent{\widecheck}{0}{mathx}{"71}

\newtheoremstyle{dotless}{}{}{\itshape}{}{\bfseries}{}
\theoremstyle{dotless}
\newtheorem{them}{Theorem}[section]	

\newtheorem{kor}[them]{Corollary}
\newtheorem{defi}[them]{Definition}
\newtheorem{remark}[them]{Remark}

\newtheorem{lem}[them]{Lemma}

\newtheorem{prop}[them]{Proposition}

\newtheorem{bsp}[them]{Example}

\title{Signature-Based Universal Bilinear Approximations for Nonlinear Systems and Model Order Reduction}
\author{Martin Redmann\thanks{University of Rostock, Institute of Mathematics, Ulmenstraße 69, 18057 Rostock, Germany, Email: {\tt
martin.redmann@uni-rostock.de}.}\and
Justus Werner\thanks{University of Rostock, Institute of Mathematics, Ulmenstraße 69, 18057 Rostock, Germany, Email: {\tt
justus.werner@uni-rostock.de}.}
}

\begin{document}
	
	\maketitle

\begin{abstract}
This paper deals with non-Lipschitz nonlinear systems. Such systems can be approximated by a linear map of so-called signatures, which play a crucial role in the theory of rough paths and can be interpreted as collections of iterated integrals involving the control process. As a consequence, we identify a universal bilinear system, solved by the signature, that can approximate the state or output of the original nonlinear dynamics arbitrarily well. In contrast to other (bi)linearization techniques, the signature approach remains feasible in large-scale settings, as the dimension of the associated bilinear system grows only with the number of inputs. However, the signature model is typically of high order, requiring an optimization process based on model order reduction (MOR). We derive an MOR method for unstable bilinear systems with non-zero initial states and apply it to the signature, yielding a potentially low-dimensional bilinear model. An advantage of our method is that the original nonlinear system need not be known explicitly, since only data are required to learn the linear map of the signature. The subsequent MOR procedure is model-oriented and specifically designed for the signature process. Consequently, this work has two main applications: (1) efficient modeling/data fitting using small-scale bilinear systems, and (2) MOR for nonlinear systems. We illustrate the effectiveness of our approach in the second application through numerical experiments.

\end{abstract}
\textbf{Keywords:} Nonlinear systems, signatures, bilinear approximations, model order reduction, non-Lipschitz vector fields

\noindent\textbf{MSC classification:} 60L10, 65C30, 93A15, 93C10

	\section{Introduction}

The study of nonlinear systems is motivated by their ability to represent complex real-world processes that linear models cannot adequately capture. In many cases, the underlying nonlinear model is unknown and only data are available, prompting the use of data-driven approaches to infer system behavior directly from observations. Moreover, when such systems are high-dimensional, their analysis becomes computationally demanding, motivating the development of reduced-order and data-efficient methods to capture essential dynamics. This paper proposes an approach to address the challenges discussed above. We develop a universal model capable of capturing the behavior of general nonlinear systems, while its unknown components can be inferred directly from observational data. The universal model can be high-dimensional, and we employ model order reduction (MOR) techniques to optimize it efficiently. Notably, high-dimensional nonlinear systems can often be accurately approximated by low-order bilinear systems using our approach, highlighting its potential for effective complexity reduction.

\subsection{Existing results and comparison to the signature approach}

This work lies at the interface of multiple disciplines. In particular, we exploit techniques from rough paths theory, MOR, and data-driven learning to establish connections between nonlinear and bilinear systems.\smallskip

Path signatures play a crucial role in the theory of rough paths \cite{Friz2020, FV10, Lyons1998}, which provides a deterministic integration framework for highly irregular paths, extending well beyond classical concepts such as Stieltjes and Young integration. For instance, it enables the pathwise definition of integrals with respect to fractional Brownian motion. Signatures also possess a universal approximation property \cite{BayerRedmann, Cuchiero23}, enabling the learning of solutions to rough or stochastic differential equations. Beyond this, they have found compelling applications in fields such as finance \cite{BayerRedmann, Cuchiero2022} and machine learning \cite{Lyons2025}, demonstrating the versatility and power of signature-based approximations. This paper presents the concept of signatures within the context of control theory, with the goal of making them accessible without extensive background in algebra or rough paths, and demonstrates their application in learning nonlinear control systems. Specifically, we show that the output of a very general nonlinear system can be approximated by a linear map of a signature, which is a solution of a bilinear system that is unstable and has a non-zero initial state. The universality of bilinear systems is known in the control community, for example, see \cite{Krener1975, Sussmann1976} for existence results. Moreover, the Carleman linearization method \cite{Rugh1981, Sastry1999} highlights this perspective by using Taylor expansions of the system's vector fields. However, the Carleman linearization requires certain regularity of the vector fields and typically ensures a small approximation error only locally. This is also true for the bilinearization proposed in \cite{Krener1975}. Another limitation of Carleman's method is that the dimension of the resulting bilinear system grows exponentially with the state dimension, making it impractical for large-scale systems. In contrast, our approach constructs a signature-based universal bilinear system that does not require smoothness of the underlying vector fields and is not limited to small neighborhoods for an accurate approximation. A key advantage is its applicability to high-dimensional systems, as the signature model grows only with the number of inputs. Finally, the computation of the signature model requires no knowledge of the underlying nonlinear system itself, relying solely on observational data of the system dynamics. It is perhaps unsurprising that the fundamental result underlying the universality of signatures is the Stone–Weierstrass theorem. This theorem has previously been applied to represent or approximate control systems \cite{Gallman1976}. However, that work employs these arguments to derive series representations of nonlinear systems in the sense of Volterra or Wiener, whereas our approach focuses on bilinearization.\smallskip

Since signature models can be applied in high-dimensional settings, they are of particular interest in the context of MOR for nonlinear systems. Applying signature-based bilinearization to the original system results in a large-scale bilinear system. MOR for such systems has been extensively studied, with balanced truncation (BT) analyzed, for instance, in \cite{Al-Baiyat1993, Benner2011}. BT offers the advantage of providing a comprehensive error analysis in the bilinear case \cite{Redmann2018, Redmann2020, Redmann2021}. However, existing approaches generally rely on asymptotic stability and zero initial states. Extensions to general initial conditions have been considered in \cite{Cao2021, Redmann2023}, but these methods still assume stability. In contrast, the signature-based bilinear system is neither stable nor initialized at zero. To address this, we study a time-limited version of BT specifically tailored to signature systems. Here, we propose Gramians and show their relations to dominant subspaces which is the fundament of the subsequent MOR procedure. Other MOR techniques, such as Krylov subspace-based or $\mathcal{H}_2$-optimal methods \cite{Benner2012, Breiten2010, Flagg2015, Zhang2002}, cannot be directly applied in this setting, but they can be adapted to handle the lack of stability and nonzero initial conditions inherent to signature models.\smallskip

Instead of relying on the bilinearization techniques proposed in this paper, MOR can be applied directly to nonlinear systems. Methods such as BT have been extended to nonlinear settings \cite{Scherpen2010, Otto2023, Scherpen1993}. Related approaches that allow for the derivation of error bounds are also available \cite{Scherpen2014, redmann2025}. However, computing the system Gramians remains highly challenging for these methods, an issue we avoid by exploiting signature approximations. Projection-based MOR techniques have also been explored by transforming nonlinear dynamics into a quadratic-bilinear form via lifting \cite{morBenG24, Gu2011, Willcox2022, Willcox2020}. Nevertheless, the analysis and computation of quadratic systems are significantly more involved than in the bilinear case. Data-driven MOR approaches \cite{Condon2004, Antoulas2018, Willcox2019, Willcox2020} have been proved effective in addressing the challenge of identifying dominant subspaces in nonlinear systems. However, some of these methods require full access to the system’s governing equations, which is often impractical. Operator inference methods \cite{Freitag2025, Willcox2021, Kramer2023, Peherstorfer2016, Peherstorfer2023} can construct reduced-order models directly from data without explicit knowledge of the system coefficients. Yet, these approaches still require the underlying model structure to be known. This contrasts with our signature-based framework, which does not assume any prior knowledge of the system dynamics.

%
% Weitere mögliche Paper (hatten wir in [] gesetzt):\\
% \cite{Gilbert1977} Functional expansions for the response of nonlinear differential systems  -- Approximation der Input-Ouput-map nichtlinearer Systeme für multiplikatives Rauschen (wie bei uns) und allgemeineren nichtlinearen Systemen. Teilweise wird der Volterra-series approach wiederentdeckt und erweitert. Könnte man bei den Volterra-series Ansätzen zitieren hat ca. 100 Zitierungen. \\
% \cite{Bai2006} A projection method for model reduction of bilinear dynamical systems -- Die machen Krylov Sachen für bilineare Systeme (single input single ouput, alles skalar), dabei approximieren sie die Kerne aus den Volterra Serien. 170 Zitierungen, könnte man denke ich bei den anderen Krylov Sachen für bilineare Systeme zitieren. \\
% \cite{Ahmad2016} Implicit Volterra series interpolation for model reduction of bilinear systems -- Geht in die Richtung des H2 Flagg Gugercin 2015 Papers. Sie leiten Optimalitätsbedingungen in H2 für das reduzierte System her und sagen diese sind erfüllt wenn eine Volterra-series-Interpolationsbedingung gilt. Hat 20 Zitierungen, könnte man bei Volterra series und H2 stuff nennen.

\subsection{Framework and story in a nutshell}
	
We study the following controlled differential equation endowed with the quality of interest $y$:
\begin{subequations}\label{original_system}
	\begin{align}
		\label{ODE}
		\dot{x}(t)&=f_0(x(t))+f(x(t))u(t),\quad x(0)=x_0,\\
		\label{eq_y}
		y(t)&=c(x(t)),
	\end{align}
	\end{subequations}
where $x:[0,T] \to \mathbb{R}^d$, $u:[0,T]\to \mathbb{R}^m$, $f_0:\mathbb{R}^d \to \mathbb{R}^d$, $f:\mathbb{R}^d \to \mathbb{R}^{d\times m}$, $y\colon[0,T] \to \mathbb{R}^{p}$ and $c\colon \mathbb R^d\to \mathbb{R}^{p}$. Moreover, let $f_0$, $f$ and $c$ be a locally Lipschitz continuous functions meaning, e.g., that for every $R>0$ there is a constant $L_{c, R}>0$ such that $\|c(x)-c(z)\| \leq  L_{c, R}\|x-z\|$ for all $x$ and $z$ with $\|x\|, \|z\|\leq R$. We further assume that the control usually satisfies $u\in L^2([0,T]; \mathbb R^m)$,  but some considerations only require $u\in L^1([0,T]; \mathbb R^m)$. Moreover, we suppose that \eqref{ODE} always has a unique global solution and will be more specific on the assumptions on $f_0$ and $f$ in Section \ref{sec_sig_approx} to ensure this.

In this work, we find a matrix $C\in \mathbb R^{p\times n}$ such that $y(\cdot)\approx CS(\cdot)$ in some sense, where $S$ is the so-called truncated signature of $\hat U(t) = (t, \int_0^t u(s)^\top\mathrm{d}s)^\top$, $t\in[0, T]$, taking values in $\mathbb R^n$. The computation of $C$ does not require the knowledge of the nonlinear system, but only data of \eqref{original_system}. The truncated signature $S$ satisfies a differential equation of the following form
	\begin{align}\label{tr_sig_ode}
		\dot{S}(t)=A_0 S(t) + \sum_{i=1}^{m} A_iS(t)u_i(t),\quad S(0)=s_0,\quad y_S(t)={C}S(t)\approx y(t),
	\end{align}
where $s_0\in\mathbb R^n$ is a particular fixed initial state, $A_0, \dots, A_m\in \mathbb R^{n\times n}$ and $u_1,\ldots,u_m\colon[0,T] \to \mathbb{R}$ are the components of $u$.
Usually, the dimension $n$ is large. Therefore, dimension reduction techniques are designed, analyzed and applied to \eqref{tr_sig_ode}. In fact, a reduced state $\widecheck S$ taking values in $\mathbb R^r$, $r\ll n$, is constructed that satisfies the following bilinear system \begin{align}\label{rom_intro}
		\dot{\widecheck S}(t)=\widecheck  A_0 \widecheck  S(t) + \sum_{i=1}^{m} \widecheck  A_i\widecheck  S(t)u_i(t),\quad {\widecheck S}(0)={\widecheck s}_0,
	\end{align}
where ${\widecheck s}_0\in\mathbb R^r$ and $\widecheck A_0, \dots, \widecheck A_m\in \mathbb R^{r\times r}$. This reduced linear system is used to approximate \begin{align*}
    y(t)\approx {C}S(t)\approx  \widecheck {C}\widecheck S(t)                                                                                                                                                                                                                                                                                                                                                                                                                                                                                 \end{align*}
for a suitable matrix $\widecheck {C}\in \mathbb R^{p\times r}$.

\section{Foundations and signatures}

% We start with results on existence and uniqueness of \eqref{ODE}. Since $u$ is potentially discontinuous we need results on so called Caratheodory-ODEs.
% \begin{lem}[Theorem 5.3. in \cite{Hale}]
% 	Consider the differential equation
% 	\begin{align}
% 	\label{CaraODE}
% 	\dot{x}(t)&=g(t,x(t)),\quad
% 	x(0)=x_0.
% \end{align}
% Let $g:[0,T] \times\mathbb{R}^d$ and arbitrary but compact $\mathcal{K}\subset [0,T] \times \mathbb{R}^d$  satisfy
% \begin{itemize}
% 	\item $g$ is measurable in $t$ for each fixed $x$
% 	\item $g$ is continuous in $x$ for each fixed $t$
% 	\item  for each $\mathcal{K}$ there is an integrable function $m_\mathcal{K}(t)$ such that
% 	\begin{align*}
% 		|g(t,x)| \leq m_\mathcal{K}(t), (t,x) \in \mathcal{K}
% 	\end{align*}
% 	\item for each $\mathcal{K}$ there is an integrable function $k_\mathcal{K}(t)$ such that
% 	\begin{align*}
% 		|g(t,x)-g(t,y)| \leq k_\mathcal{K}(t)|x-y|, \quad (t,x),(t,y) \in \mathcal{K}
% 	\end{align*}
% \end{itemize}
% then \eqref{CaraODE} has a unique global solution in the sense of Caratheodory.

%\end{lem}
We briefly link \eqref{ODE} to Stieltjes differential equations requiring the following definition.
	\begin{defi}[Bounded Variation]
A function $f:[0,T] \to \mathbb{R}^m$ is of bounded variation
		if 
		\begin{align*}
			\|f\|_{BV}:=\sup_{P\in \mathcal{P}} \sum_{i=0}^{n_P-1}\|f(t_{i+1})-f(t_i)\|<\infty,
		\end{align*}
where %the supremum is taken over the set
$\mathcal{P}=\{P=\{t_0,\ldots, t_{n_P}\}: P \text{ is a partition of }[0,T]\}$.

	\end{defi}
	Note that $\|\cdot\|_{BV}$ is only a seminorm. We deal with a particular bounded variation path $U$ considered in the next lemma.
	\begin{lem}\label{ineqBVL1}
	Let $u \in L^1([0,T];\mathbb{R}^m)$ and $U(\cdot)=\int_{0}^{\cdot} u(t) \, \mathrm{d}t$. The absolutely continuous function $U$ is of bounded variation. More precisely, it holds that \begin{align*}
			\|U\|_{BV} \leq \|u\|_{L^1([0, T])}.
		\end{align*}
	\end{lem}
	\begin{proof}
Proofs of the above statements can be found in \cite[Lemma 3.34]{Folland} and \cite[Lemma 5.2.9]{Heil}.
	\end{proof}
% 	\begin{defi}[Absolute Continuity (wird gerade nicht genutzt (?))]
% 		A function $f:[0,T] \to \mathbb{R}^d$ is absolutely continuous on $[0,T]$ if for $\varepsilon>0$ there is a positive number $\delta>0$ such that every finite sequence of pairwise disjoint sub-intervals $(x_k,y_k)$ of $[0,T]$ with $x_k<y_k$ satisfies
% 		\begin{align*}
% 			\sum_{k=1}^{N}y_k-x_k<\delta
% 		\end{align*}
% 		then
% 		\begin{align*}
% 			\sum_{k=1}^{N}\|f(y_k)-f(x_k)\|<\varepsilon.
% 		\end{align*}
% 	\end{defi}
A Stieltjes differential equation is of the form
	\begin{align}
		\label{Stieltjes-DE}
		\mathrm{d}x(t)=f_0(x(t)) \,\mathrm{d}t+f(x(t))\,\mathrm{d}U(t),\quad
		x(0)=x_0,
	\end{align}
where $U$ is a bounded variation path. Now, setting $U\coloneqq\int_0^\cdot u(t) \mathrm{d}t$, equation \eqref{Stieltjes-DE} is equivalent to \eqref{ODE}. This becomes clearer applying the following lemma. The Stieltjes differential equation perspective can be beneficial as  \eqref{Stieltjes-DE} is covered by rough differential equations \cite{FV10} (smooth case).
% 	\begin{proof}´
% 		Antiderivatives are absolute continuous (\cite{Folland} Korollar 3.33, Theorem 3.35), therefore $U$ is of bounded variation.  \\
% 		If $\Pi=\{0=x_0<\cdots <x_n=T \}$ is a partition of $[0,T]$, then
% 		\begin{align*}
% 			\sum_{j=1}^n
% 			\|U(x_j)-U(x_{j-1})\|=\sum_{j=1}^n \|\int_{x_{j-1}}^{x_{j}}u(t)\, \mathrm{d}t\|\leq \sum_{j=1}^n \int_{x_{j-1}}^{x_{j}}\|u(t)\| \, \mathrm{d}t=\|u\|_{L^1}
% 		\end{align*}
% 		Taking the supremum of all partitions concludes the proof.
% 	\end{proof}
% 		I think we even have equality (see https://math.stackexchange.com/questions/1516309/total-variation-of-a-continuously-differentiable-function) since $u$ is integrable but I did not find a peer reviewed proof which is citable.
% 	\begin{lem}
% 		Let $u \in L^1([0,T])$ and $U(\cdot)=\int_{0}^{\cdot} u(t) \, \mathrm{d}t$. Then
% 		\begin{align*}
% 			\|U\|_{BV} =\|u\|_{L^1}.
% 		\end{align*}
% 	\end{lem}
% 	\begin{proof}
% 		?
% 	\end{proof}
% 	\begin{lem}[\cite{Natanson} Theorem 1 p. 261]
% 		\label{TriangleIneq}
% 		Let $f$ be continuous and $U$ of bounded variation, then
% 		\begin{align*}
% 			\|\int_{0}^T f(t) \,\mathrm{d}U(t)\| \leq \max_{t \in[0,T]} \|f\|\cdot \|U\|_{BV}.
% 		\end{align*}
% 	\end{lem}
	
		\begin{lem}[{\cite[Theorem 3]{Natanson}}]
			\label{KonvRiemannStieltjes}
		Let $f$ be continuous, $u \in L^1([0,T];\mathbb{R}^m)$ and $U(\cdot)=\int_0^{\cdot}u(t) dt$. Then,
		\begin{align*}
			\int_{0}^T f(t) \,\mathrm{d}U(t)\ =	\int_{0}^T f(t) u(t)\,\mathrm{d}t.
		\end{align*}
	\end{lem}

% 	\begin{lem}[Gronwall]
% 		Let $\alpha,\beta$ and $f$ be real-valued functions defined on $[0,T]$. Assume that $\beta$ and $f$ are continuous and that the negative part of $\alpha$ is integrable on every compact subinterval of $[0,T]$.\\
% 		If $\beta$ is non-negative, $\alpha$ is non-decreasing and if $f$ satisfies the inequality
% 		\begin{align*}
% 			f(t)\leq\alpha(t)+\int_{0}^{t} \beta(s) f(s) \, \mathrm{d}s
% 		\end{align*}
% 		for all $t \in [0,T]$, then
% 		\begin{align*}
% 			f(t) \leq \alpha(t)\exp \Big(\int_0^t \beta(s) \, \mathrm{d}s\Big).
% 		\end{align*}
% for all $t \in [0,T]$.
% 	\end{lem}

Next we introduce signatures of bounded variation paths which are, e.g., drivers of equations like \eqref{Stieltjes-DE}. The theory of signatures is developed beyond the scope of this paper and plays for example a crucial role in the construction of rough path theory \cite{Friz2020, FV10, Lyons1998}. While signatures are often introduced exploiting many algebraic concepts and rough paths theory, we give some insights to these important objects without requiring deeper knowledge in the above mention fields.
\begin{defi}[(Truncated) Signature]\label{def_sig}
Let $X\colon[0,T]\to \mathbb{R}^m$ be a continuous path of bounded variation. Then, the signature $	S_{\cdot,\cdot}(X)\colon [0,T]\times [0,T] \to  \prod_{k=0}^\infty \mathbb{R}^{m^k}$ over an interval $[a,b]$ with $0\leq a <b \leq T$ is given as the infinite vector of Stieltjes integrals
\begin{align*}
S_{a,b}(X)	=\left(\begin{matrix}1 & {X^{(1)}_{a,b}}^\top & {X^{(2)}_{a,b}}^\top  & \ldots \end{matrix}\right)^\top
\end{align*}
with the iterated integrals $X^{(k)}_{a,b}\in \mathbb{R}^{m^k}$ defined by
\begin{align*}
X^{(k)}_{a,b}=\underset{a<s_1<\cdots < s_{k}<b}{\int \cdots \int}\mathrm{d}X(s_1) \otimes \cdots \otimes \mathrm{d}X({s_{k}}),
\end{align*}
where $\cdot\otimes\cdot$ denotes the Kronecker product of two vectors/matrices. Moreover, we call
\begin{align*}
S^N_{a,b}(X)=\left(\begin{matrix}1 & {X^{(1)}_{a,b}}^\top & {X^{(2)}_{a,b}}^\top  & \ldots & {X^{(N)}_{a,b}}^\top \end{matrix}\right)^\top
		\end{align*}
truncated signature of order $N\in\mathbb N$, where  $S^N_{a,b}(X) \in \prod_{k=0}^N \mathbb{R}^{m^k}= \mathbb{R}^{n}$ with $n=\sum_{k=0}^{N}m^k=\frac{m^{N+1}-1}{m-1}$.
	\end{defi}
While (truncated) signatures are mostly introduced as paths taking values in a (truncated) tensor algebra, the usage of the Kronecker product in Definition \ref{def_sig} avoids this. In the following example, we see that the iterated integrals of $U=\int_{0}^{\cdot} u(s) \, \mathrm{d}s$ are just classical integrals of $u$.
	\begin{bsp}
Let $X=U$ in Definition \ref{def_sig}. By Lemma \ref{KonvRiemannStieltjes}, we have
\begin{align*}
	U_{a,b}^{(k)}=\int_a^b  \int_{a}^{s_{k}} \cdots \int_a^{s_2} u(s_1)\otimes  \cdots  \otimes u(s_{k}) \, \, \mathrm{d}s_1 \cdots \mathrm{d}s_{k}.
\end{align*}
	\end{bsp}
Let $O_{n_1,n_2}\in \mathbb{R}^{n_1 \times n_2}$ denote the zero matrix. Using this notation, we show that the truncated signature satisfies a linear Stieltjes differential equation.
\begin{prop}\label{SigODE}
	Let $e_i\in \mathbb{R}^{m}$ be the $i$th unit vector, $X$ a continuous bounded variation path taking values in $\mathbb R^m$ and $S^N(X)$ its truncated signature of order $N\in\mathbb N$. Let us introduce the dimensions  $n=\frac{m^{N+1}-1}{m-1}, \, \tilde{n}=\frac{m^{N}-1}{m-1}$ as well as matrices $A_i \in \mathbb{R}^{n \times n}$ via
	\begin{align*}
		A_i=\begin{bmatrix}  \tilde{A}_i, {O}_{n,n-\tilde{n}}	 \end{bmatrix},\quad \tilde{A}_i=\begin{bmatrix}
			{O}_{1,\tilde{n}}\\	\blkdiag(\underbrace{e_{i}, \ldots,e_{i}}_{\tilde{n}\text{-times}}  )
		\end{bmatrix}.
	\end{align*}
	Then, $S^N(X)$ satisfies

\begin{align}\label{stieltjes_de}
			{S}^N_{0,t}(X)=\mathbf e_1+\sum_{i=1}^{m}\int_{0}^{t}A_i{S}^N_{0,s}(X)\mathrm{d}{X}_i(s),
		\end{align}
% 	\begin{align}
% 		\mathrm{d}{S}^N_{0,t}(X)=\sum_{i=1}^{m}A_i{S}^N_{0,t}(X)\mathrm{d}{X}_i(t), \quad {S}^N_{0,0}(X) = \mathbf e_1,
% 	\end{align}
where $\mathbf e_1$ is the first unit vector in $\mathbb R^n$.
\end{prop}
	\begin{proof}
		Let $0\leq k \leq N-1$ and $X^{(0)}\equiv 1$, then
		\begin{align*}
			X^{(k+1)}_{0,t}&=\int_{0}^{t} 
			X^{(k)}_{0,s}
			\otimes \mathrm{d}  X(s)=
			\sum_{i=1}^m\int_{0}^{t} 
			X^{(k)}_{0,s}
			\otimes e_i \, \mathrm{d}  X_i(s)
			\\&=	\sum_{i=1}^m\int_{0}^{t}  \blkdiag(\underbrace{e_{i}, \ldots,e_{i}}_{m^k\text{-times}}) \, 
			X_{0,t}^{(k)}  \, \mathrm{d}  X_i(s)
		\end{align*}
representing $X^{(k)}_{0,s}
\otimes e_i $ by a matrix-vector multiplication. The first entry $(S_{0,\cdot}^N(X))_1=X^{(0)}=1$ on the left-hand side of \eqref{stieltjes_de} also needs to be replicated on the right-hand side which is done by integrating a zero and restoring the first component via the initial condition. This is taken into account by adding a first line of zeros to the matrices $\tilde{A}_i$ and choosing the initial state to be $\mathbf e_1$. Further, to include $X^{(N)}$ on the right-hand side of \eqref{stieltjes_de} we add more zeros to $\tilde{A}_i$ to restore the correct dimensions. The number of zero columns added is  the dimensions of $X^{(N)}$. Hence, we see that \eqref{stieltjes_de} holds.
This concludes the proof.
	\end{proof}
\begin{remark}\label{remark_sig_bil}
	An iteration to compute the matrices $A_i$ is given in \cite{BayerRedmann}. Applying Proposition \ref{SigODE} to $\hat{U}(t)=(t,U(t)^\top)^\top$ yields the bilinear control system \eqref{tr_sig_ode}. The truncated signature $S^N(\hat U)$ plays a crucial role in our algorithm later, as we can use it to approximate the output in \eqref{eq_y}.
\end{remark}

	 We note that the zeroth entry of the signature is $1$ by convention. This is consistent with the literature on signatures. The signature $S_{a,b}(X)$ captures certain information about the path $X|_{[a,b]}$. For instance, the first level $X_{a,b}^{(1)}$ encodes the increment of $X$ on the interval $[a,b]$. Similar interpretations exist for higher levels of the signature $X_{a,b}^{(k)}$ for $k \geq 2$; for example the second level is related to the L\'evy area. A natural question is to what extent the path $X|_{[a,b]}$ can be reconstructed from its signature $S_{a,b}(X)$. We begin by observing that, $S_{a,b}(X)=S_{a,b}(X+c)$ for any constant vector $c$. This invariance property follows directly from the definition of the Stieltjes integral. Paths with the same initial value are uniquely determined by the signature, up to so-called tree-like equivalence \cite{Hambly2010}. However, if at least one component of $X$ is monotone, then $X$ is uniquely determined by $S_{a,b}(X)$ up to an initial value. In our setting, this monotone variable is the time component. This gives rise to the following lemma.
	 
	 \begin{lem}[Uniqueness of the Signature]
	\label{UniquenesSignature}
	 	Let $X,Y\colon[0,T] \to \mathbb{R}^m$ be bounded variation paths and $X(0)=Y(0)=c\in \mathbb{R}^m$. Set $\hat{X}(t) \coloneqq(t,X(t)^\top)^\top$ and $\hat{Y}(t) \coloneqq(t,Y(t)^\top)^\top$, $t\in [0, T]$. Then, $S_{0,T}(\hat X)=S_{0,T}(\hat Y)$ if and only if $X(t)=Y(t)$ for each $t \in [0,T]$.
	 \end{lem}
	 \begin{proof}
	We refer to  Lemma 2.6 in  \cite{Cuchiero2022} for a proof.
	 \end{proof}
	 
	 A more sophisticated, rough path version of this lemma can for example be found in \cite[Corollary 3.10]{Cuchiero23} or in \cite{Boedihardjo2015}.

	\section{Signature approximations of continuous maps}\label{sec_sig_approx}
We interpret the solution of \eqref{ODE} as a function $u\mapsto x(u)$ and use the following notation $x(t; u):=\big(x(u)\big)(t)$, $t\in[0, T]$. We will show that continuous maps related to $u \mapsto x(\cdot;u)$ can be approximated by linear maps of truncated signatures, see in Section \ref{3_2}. Therefore, the continuity of the solution map $u \mapsto x(\cdot;u)$ is an essential property that we discuss in Section \ref{3_1}. It is important to notice that this is a classical topic of interest in ODE or control theory. The parameter-dependent ODE case with $u \in \mathbb{R}^m$ is discussed in \cite{Hale}. Fitting our framework, where $u$ takes values in a function space, continuity results can be found in \cite{Artstein1975, Neustadt1970}. A general control-theoretic framework that also includes \eqref{ODE} is provided in \cite{Sussmann1976}. In this section, we point out that under the assumptions we make to ensure the existence of a global solution of \eqref{ODE}, we obtain more than continuity. In fact, $u \mapsto x(\cdot;u)$ is locally Lipschitz. This can also be concluded from \cite{FV10}, where it is shown that $U \mapsto x(\cdot;U)$ in \eqref{Stieltjes-DE} and hence, by Lemma \ref{ineqBVL1}, also $u \mapsto x(\cdot;u)$ is locally Lipschitz. However, \cite{FV10} assumes bounded and smooth vector fields $f_0$ and $f$ given a general bounded variation driver $U$, which, in the special case of absolutely continuous drivers, imposes unnecessarily assumptions on the vector fields. We do not need any smoothness or boundedness in our arguments below. For the reader's convenience, we present proofs of the local Lipschitz continuity of $u \mapsto x(\cdot;u)$ within our framework.

\subsection{Locally Lipschitz continuous maps of the control process $u$}\label{3_1}
We usually assume $u\in L^q([0,T];\mathbb{R}^m)$ ($q=1, 2$) like in the following lemma which is required for the continuity proof.
% 	\begin{align}
% 		\label{z-ODE}
% 		\dot{z}(t)=f_0(z(t))+f(z(t))v(t )
% 	\end{align}
% 	where the vector fields $f$ are the same as in \ref{ODE} and $v$ has the same assumptions as $u$.
% 	\begin{defi}
% 			For $p\geq 1$ we define $\mathcal{N}^R_p([0,T] ; \mathbb{R}^m)\coloneqq \{u:[0,T] \to \mathbb{R}^m: \|u\|_{L^p([0,T])}\leq R \}$ with $R>0$.
% 	\end{defi}

	\begin{lem}
		\label{Hilfslemma}
		Let $f_0$ and $f$ in \eqref{ODE} be such that a unique global solution exists. Let either \begin{itemize}
\item[(a)] $\|f_0(x)\|+\|f(x)\|\leq \mathcal C(1+\|x\|)$ and $\|u\|_{L^1([0,T])}\leq R$ or
\item[(b)] $2\langle x, f_0(x)\rangle + \big\|f(x)\big\|^2_F\leq \mathcal C(1+\|x\|^2)$ and $\|u\|_{L^2([0,T])}\leq R$
	\end{itemize}
	be satisfied for some constant $\mathcal C$ and all $x\in\mathbb R^d$,
 where $R>0$ is arbitrary. Given a continuous $f$ and the solution $x=x(t; u)$ of \eqref{ODE}, we find that
		\begin{align*}
		\sup_{t \in [0,T]}	\|f(x(t; u))\|\leq K_R,
		\end{align*}
		where $K_R$ depends on $R$ but is independent of $u$.
	\end{lem}
\begin{proof}
 We moved this proof to Appendix for a better readability of the paper, see Section \ref{sec_proof_Lemma0}.
\end{proof}

\begin{remark}\label{rem_link_sde}
The controlled differential equation in \eqref{ODE} can be linked to stochastic differential equations (SDEs). Setting $u(t)=\frac{\mathrm{d}W(t)}{\mathrm{d}t}$, where $W$ is an $m$-dimensional Wiener process would formally lead to the referred SDE case. However, that control is not a function, but a distribution. Suppose that local Lipschitz continuity of $f_0$ and $f$ as well as the one-sided linear growth condition in Lemma \ref{Hilfslemma}, that is, \begin{align}\label{lin_growth_cond}
2\langle x, f_0(x)\rangle + \big\|f(x)\big\|^2_F\leq \mathcal C(1+\|x\|^2)
\end{align}
for some constant $\mathcal C$ and all $x\in\mathbb R^d$ are given. Then, it is known that an SDE ($u(t)=\frac{\mathrm{d}W(t)}{\mathrm{d}t}$) has a unique global solution according to \cite[Theorem 3.5]{mao}. In fact, this result can immediately be transferred to \eqref{ODE} if $u\in L^2([0, T]; \mathbb R^m)$. This is due to estimate \eqref{est_square} which yields \begin{align}\label{bound_ODE_by_SDE}
2  \langle x,   f_0(x)\rangle+2 \langle x,   f(x)u(t)\rangle \leq 2  \langle x,   f_0(x)\rangle+\big\|f(x)\big\|^2_F + \|x\|^2\|u(t)\|^2
\end{align}
for all $x\in\mathbb R^d$ and $t\in [0, T]$. The left-hand side of \eqref{bound_ODE_by_SDE} occurs in an existence and uniqueness proof of \eqref{ODE} and can be bounded by an expression $2  \langle x,   f_0(x)\rangle+\big\|f(x)\big\|^2_F$ relevant in the SDE case. Therefore, the arguments of Mao \cite{mao} can be used here as well. In fact, requiring only locally Lipschitz continuous vector fields combined with \eqref{lin_growth_cond} is an enormous gain in comparison to the global Lipschitz case, where a unique global solution exists if $u\in L^1([0, T]; \mathbb R^m)$. This is because certain polynomials can now be covered. We refer to Example \ref{ex1} for concrete choices of $f_0$ and $f$ not being globally Lipschitz. Let us finally mention that a unique local solution of \eqref{ODE} exists under the so-called Caratheodory conditions, see for instance \cite[Theorem 5.3.]{Hale}. However, we operate with global solutions only.
\end{remark}
In the following lemma we show that the solution $x$ depends continuously on the control $u$. In particular, $x$ is locally Lipschitz in $u$ under global Lipschitz assumptions on $f_0$ and $f$.
	\begin{them}[$L^1$-case]
		\label{Lemma1}
		Given an arbitrary constant $R>0$ and $u,v\in L^1([0,T] ; \mathbb{R}^m)$ with $\|u\|_{L^1([0,T])}, \|v\|_{L^1([0,T])}\leq R$. Suppose that the coefficients $f_0$ and $f$ are globally Lipschitz meaning that $\|f_0(x)-f_0(z)\|+\|f(x)-f(z)\|\leq L\|x-z\|$ for some $L>0$ and all $x, z\in\mathbb R^d$. Then, we have
		\begin{align*}
			\sup_{t \in [0,T]}\|x(t; u)-x(t; v)\| \leq K \|u-v\|_{L^1([0,T])},
		\end{align*} 
where $K>0$ is a suitable constant depending on $R$.
	\end{them}
	\begin{proof}
The proof can be found in Section \ref{sec_proof_Lemma1}.
	\end{proof}
While the result of Theorem \ref{Lemma1} might be known, the generalization to one-sided Lipschitz vector field might be less familiar. The statement of Theorem \ref{Lemma1} still holds true with less assumptions on the vector fields $f_0$ and $f$ if $u \in L^2([0,T] ; \mathbb{R}^m)$.
		\begin{them}[$L^2$-case]
			\label{Lemma2}
			Given an arbitrary constant $R>0$ and $u,v\in L^2([0,T]; \mathbb{R}^m)$ with $\|u\|_{L^2([0,T])}, \|v\|_{L^2([0,T])}\leq R$. Suppose that the coefficients $f_0$ and $f$ are locally Lipschitz continuous. In addition, let us assume that these coefficients satisfy a one-sided global Lipschitz property, that is,
			\begin{align*}
			2\,\langle x-z,f_0(x)-f_0(z)\rangle+\|f(x) -f(z)\|_F^2 \leq L\|x-z\|^2
		\end{align*}
		holds for some constant $L$ and all $x, z\in\mathbb R^d$. Then, we have
			\begin{align*}
			\sup_{t \in [0,T]}\|x(t; u)-x(t; v)\| \leq K \|u-v\|_{L^2([0,T])},
		\end{align*}
		where $K$ is a suitable constant depending on $R$.
		\end{them}
		\begin{proof}
The proof is stated in Section \ref{sec_proof_Lemma2}.
		\end{proof}
\begin{bsp}\label{ex1}
Below, let us construct an example that satisfies the requirements of Theorem \ref{Lemma2} and hence also the ones of Lemma \ref{Hilfslemma} case (b).
Suppose that $m=1$. We set $f_0(x) = \mathcal A x - x^{\circ 3}$ and $f(x) = x^{\circ 2}$, where $\mathcal A\in\mathbb R^{d\times d}$ and $\circ$ indicates component-wise powers.  Both $f_0$ and $f$ are locally Lipschitz continuous, because they are polynomials and hence have continuous derivatives. Moreover, we observe that
\begin{align*}
 -2\langle x- z,  x^{\circ 3} - z^{\circ 3}\rangle&=-2\sum_{i=1}^d (x_i^4 +z_i^4 - z_i x_i^3-x_i z_i^3)
 =-2\sum_{i=1}^d (x_i-z_i)^2 (x_i^2 +z_i^2 + z_i x_i)\\&\leq -\sum_{i=1}^d (x_i-z_i)^2 (x_i +z_i)^2
\end{align*}
for $x=(x_1, \dots, x_d)^\top$ as well as $z=(z_1, \dots, z_d)^\top$ and further find that\begin{align*}
\|x^{\circ 2}-z^{\circ 2}\|^2 =\sum_{i=1}^d(x_i^2-z_i^2)^2= \sum_{i=1}^d(x_i-z_i)^2(x_i+z_i)^2.
               \end{align*}
Consequently, we have
\begin{align*}
			2\,\langle x-z,f_0(x)-f_0(z)\rangle+\|f(x) -f(z)\|^2 \leq 2\langle x-z, \mathcal A(x-z)\rangle\leq L\|x-z\|^2
		\end{align*}
for a suitable constant $L$. Therefore, the assumptions of Theorem \ref{Lemma2} are satisfied showing that interesting polynomial nonlinearities are covered.
\end{bsp}

\begin{remark}\label{rem1}
\begin{itemize}
 \item The reason why local Lipschitz continuity for $f_0$ and $f$ is assumed in Theorem \ref{Lemma2} is the need for a global solution to exist, see Remark \ref{rem_link_sde}.
\item The local Lipschitz property of the map $u\mapsto x(\cdot, u)$ established in Theorems  \ref{Lemma1} and \ref{Lemma2} directly transfers to $u\mapsto y(\cdot, u)$. Let $R>0$ be arbitrary and $\|u\|_{L^q([0,T])}, \|v\|_{L^q([0,T])}\leq R$ for either $q=1$ if the assumptions of Theorem \ref{Lemma1} hold or $q=2$ if the conditions of Theorem \ref{Lemma2} are satisfied. Then, we know by  \eqref{x_in_ball} and \eqref{x_in_ball2} that $x(\cdot, u)$ and $x(\cdot, v)$ take values in a ball with finite radius independent of the particular control. Now that $c$ is Lipschitz continuous on that ball, we find a constant $L_c>0$ such that \begin{align*}
			\sup_{t \in [0,T]}\|y(t; u)-y(t; v)\|&=\sup_{t \in [0,T]}\|c(x(t; u))-c(x(t; v))\|\leq L_c \sup_{t \in [0,T]}\|x(t; u)-x(t; v)\| \\&\leq L_c K \|u-v\|_{L^q([0,T])}
		\end{align*}
for $\|u\|_{L^q([0,T])}, \|v\|_{L^q([0,T])}\leq R$ and $q=1, 2$ (depending on the assumptions on the vector fields).
\end{itemize}
\end{remark}
The next section addresses approximations of continuous functionals and maps of the control $u$ using signatures.

\subsection{Universal approximations by the signature}\label{3_2}

		We further need to consider the linear functionals on $\prod_{k=0}^\infty \mathbb{R}^{m^k}$ which we denote by $\mathcal W_m\coloneqq\bigoplus_{k=0}^\infty \mathbb{R}^{m^k}$.
		This corresponds to assigning a weight to each individual iterated integral, which are collected in the signature. In the common literature this is done via a pairing of elements from dual spaces and the signature. In our setting, namely the real, finite dimensional case, this boils down to an inner product. The weight $\ell \in \mathcal{W}_{m}$ is in common literature named as word and per definition unequal to $0$ in only finitely many components.
		\begin{lem}
			\label{Hilfslemma2}
		Let $q \geq 1$ and $x \in L^q([0,T];\mathbb{R}^m)$, further let $X(t)\coloneqq \int_{0}^{t} x(s) \, \mathrm{d}s$ and $\hat X(t)\coloneqq (t,X(t)^\top)^\top$, $t\in[0, T]$. Then, the map $h\colon L^q([0,T];\mathbb{R}^m) \to \mathbb{R}, \, {x} \mapsto \langle \ell, S_{0,T}(\hat{X})  \rangle$ is continuous for every $\ell \in \mathcal{W}_{m+1}$.
	\end{lem}
	\begin{proof}
% 		We proof the Lemma for $\ell$ which only have one component different from $0$, which w.l.o.g. we assume to be equal to $1$. Since the linear combination of continuous functions is continuous again the statement follows for all $\ell \in \mathcal{W}_{m+1}$.\\
Suppose that  $x^i \in L^q([0,T];\mathbb{R}^m)$ and that $\hat{X}^i$ is the time extended version of ${X}^i(\cdot)\coloneqq \int_{0}^{\cdot } x^i(t) \, \mathrm{d}t$ for $i=1, 2$. Since $\ell\in\mathcal W_{m+1}$ has only finitely many entries different from zero, there is an order $N\in\mathbb N$ and an $l\in \mathbb R^n$ with $n=n(N, m)$, such that $\langle\ell, S_{0,T}({\hat{X}^i})  \rangle=\langle l, S^N_{0,T}({\hat{X}^i})  \rangle$.
%$\langle\ell, S_{0,T}({\hat{X}^i})  \rangle =\int_{0}^{T} \int_{0}^{r_{n-1}}\cdots \int_{0}^{r_{{1}}}\mathrm{d}\hat{X}^i_{j_0}(r_0)\,\mathrm{d}\hat{X}^i_{j_1}(r_1) \cdots \mathrm{d}\hat{X}^i_{j_{n-1}}(r_{n-1})$ for $i=1,2$ be an arbitrary $n$-times iterated integral.
Now, it follows by Proposition 7.63 in \cite{FV10} and Lemma \ref{ineqBVL1} that
		\begin{align*}
			\vert\langle\ell, S_{0,T}({\hat{X}^1})- S_{0,T}({\hat{X}^2})  \rangle\vert &\leq \|l\| \|S^N_{0,T}({\hat{X}^1})- S^N_{0,T}({\hat{X}^2})  \|\leq K \|\hat{X}^1-\hat{X}^2\|_{BV} \\
			&\leq K \|x^1-x^2\|_{L^1([0,T])}\leq \tilde K\|x^1-x^2\|_{L^q([0, T])} \to 0
		\end{align*}
for ${\|x^1-x^2\|_{L^q([0, T])} \to 0}$, where $K, \tilde K>0$ are suitable constants.
	\end{proof}
		Next we introduce a universal approximation theorem (UAT) in the spirit of the classical UAT from rough path theory. We extend the integral $U(t)=\int_0^t u(s)\, \mathrm{d}s$, $t\in[0, T]$, of the control $u$ by including a time component leading to $\hat{U}(t) \coloneqq (t,U(t)^\top)^\top$. 
		\begin{them}[Universal Approximation Theorem]
			\label{UAT}
			Given $q\geq 1$, let $\mathcal{K} \subset L^q([0,T];\mathbb{R}^{m})$ be compact and $g \colon L^q([0,T];\mathbb{R}^{m}) \to \mathbb{R}$ be continuous. Then, for every $\varepsilon>0$, there exists an $\ell \in \mathcal W_{m+1}$, such that
			\begin{align*}
				\sup_{{u} \in \mathcal{K}} |g({u})-\langle \ell, S_{0,T}(\hat{U})  \rangle|< \varepsilon.
			\end{align*} 
		\end{them}
		\begin{proof}
		The proof is done via an application of the Stone-Weierstrass theorem. We need to prove that 
		\begin{align*}
			\mathcal{A}\coloneqq \{{u} \mapsto \langle \ell, S_{0,T}(\hat{U})  \rangle | \ell \in \mathcal W_{m+1}, {u} \in L^q([0,T];\mathbb{R}^{m}) \}
		\end{align*}
		is a subalgebra of $C(L^q([0,T];\mathbb{R}^{m});\mathbb{R})$ which is point-separating and contains a non-zero constant function. \\From Lemma \ref{Hilfslemma2} the map $u \mapsto \langle \ell, S_{0,T}(\hat{U})  \rangle$ is continuous for every $\ell \in \mathcal W_{m+1}$ and therefore $\mathcal{A}$ is a subset of $C(L^q([0,T];\mathbb{R}^{m});\mathbb{R})$. Now, by Lemma \ref{UniquenesSignature}, we know that the signature of the time-extended path $\hat U$ uniquely determines $U$. Moreover, $U$ uniquely determines the control $u$ in $ L^q([0,T];\mathbb{R}^{m})$, $q\geq 1$. This means that if we have two controls $u$ and $v$ and the associated integrals $U(t)=\int_0^t u(s)\, \mathrm{d}s$ and $V(t)=\int_0^t v(s)\, \mathrm{d}s$, then $U(t)=V(t)$ for all $t\in [0, T]$ implies that $u(t)=v(t)$ for almost all $t\in [0, T]$. Therefore, $\mathcal A$ is point-separating.
$\mathcal A$ contains the constant non-zero function ${u} \mapsto 1$ which is obtained by choosing $\ell=(1, 0, 0, \dots)$.
This leaves us to prove that $\mathcal{A}$ is an algebra, which means that for $\ell_1, \ell_2 \in \mathcal W_{m+1}$ there exists an $\ell_3 \in \mathcal W_{m+1}$, such that
		\begin{align}\label{to_be _shown}
			\langle \ell_1, S_{0,T}(\hat{U})  \rangle\langle \ell_2, S_{0,T}(\hat{U})  \rangle=\langle \ell_3, S_{0,T}(\hat{U})  \rangle.
		\end{align}
		This boils down to the question whether the product of iterated integrals is a linear combination of iterated integral again. This can be proved by induction over the total number of integrals. It is enough to consider $\ell_1, \, \ell_2 \in \mathcal W_{m+1}$, which only contain one component different from $0$ and w.l.o.g we assume these coefficients to be equal to $1$. We begin the induction considering the product of first order iterated integrals. The product rule yields
		\begin{align*}
		\int_{0}^{t}\mathrm{d}\hat U_i(r_1)\cdot \int_{0}^{t}\mathrm{d}\hat U_j(r_1)=\int_{0}^t\int_{0}^{r_1}\mathrm{d}\hat U_i(r_2)\mathrm{d}\hat U_j(r_1)+\int_{0}^{t}\int_{0}^{r_1}\mathrm{d}\hat U_j(r_2)\mathrm{d}\hat U_i(r_1).
	\end{align*}
Setting $t=T$ leads to an identity like in \eqref{to_be _shown}.
For the induction step define
\begin{align*}
Z(t)&=\langle \ell_1, S_{0,t}(\hat{U})  \rangle=\int_{0}^{t} \int_{0}^{r_{q-1}}\cdots \int_{0}^{r_{1}}\mathrm{d}\hat U_{i_0}(r)\,\mathrm{d}\hat U_{i_1}(r_1) \cdots \mathrm{d}\hat U_{i_{q-1}}(r_{q-1}), \\Y(t)&=\langle \ell_2, S_{0,t}(\hat{U})  \rangle=\int_{0}^{t} \int_{0}^{v_{k-1}}\cdots \int_{0}^{v_{1}}\mathrm{d}\hat U_{j_0}(v)\,\mathrm{d}\hat U_{j_1}(v_1) \cdots \mathrm{d}\hat U_{j_{k-1}}(v_{k-1})
\end{align*}
with total number of integrals equal to $q+k$. Then, by the product formula, we get
		\begin{align*}
			Y(t)Z(t)&=\int_0^t Y(s) \, \mathrm{d}{Z}(s)+\int_0^t Z(s) \, \mathrm{d}{Y}(s)\\
			&=\int_0^t  \Big(Y(s) \int_{0}^{s}\int_{0}^{r_{q-2}}\cdots \int_{0}^{r_{1}}\mathrm{d}\hat U_{i_0}(r)\,\mathrm{d}\hat U_{i_1}(r_1) \cdots \mathrm{d}\hat U_{i_{q-2}}(r_{q-2}) \Big) \, \mathrm{d}\hat {U}_{i_{q-1}}(s)\\&\quad +\int_0^t  \Big(Z(s)  \int_{0}^{s} \int_{0}^{v_{k-2}} \cdots \int_{0}^{v_{1}}\mathrm{d}\hat U_{j_0}(v)\,\mathrm{d}\hat U_{j_1}(v_1) \cdots \mathrm{d}\hat U_{j_{k-2}}(v_{k-2})\Big) \,\mathrm{d}\hat {U}_{j_{k-1}}(s)\\
			&=\int_0^t  \langle \tilde \ell_3, S_{0,s}(\hat{U})  \rangle  \, \mathrm{d}\hat {U}_{i_{q-1}}(s) +\int_0^t   \langle \bar \ell_3, S_{0,s}(\hat{U})  \rangle \,\mathrm{d}\hat {U}_{j_{k-1}}(s)=\langle  \ell_3, S_{0,t}(\hat{U})  \rangle,
		\end{align*}
where we used the induction hypothesis twice to the product of iterated integrals with total number of integrals equals $q+k-1$ for some $\tilde \ell_3, \bar \ell_3, \ell_3 \in\mathcal W_{m+1}$.
		\end{proof}
By Theorems \ref{Lemma1} and \ref{Lemma2} we know that $u \mapsto x(\cdot;u)$ is locally Lipschitz continuous. By Remark \ref{rem1}, the same is true for $u \mapsto y(\cdot;u)$ under the assumptions we made on $f_0$ and $f$. Therefore, a potential functional $g$ of interest could be the $i$th component of $x$ or $y$ evaluated at $T$, i.e., $g(u)=x_i(T; u)$ or $g(u)=y_i(T; u)$.  Theorem \ref{UAT} immediately guarantees the existence of a function $\ell_i \colon [0,T] \to \mathcal W_{m+1}$, such that
				\begin{align*}
			y_i(t; u)\approx \langle \ell_i(t), S_{0,t}(\hat{U})  \rangle.
		\end{align*}
However, we would rather like to approximate the functions $x$ or $y$ on $[0, T]$ with a functional being independent of time.
Therefore, the next non-trivial step is to find an $\ell_i\in\mathcal W_{m+1}$ giving us
		\begin{align}\label{intu_UAT}
y_i(t; u)\approx \langle \ell_i, S_{0,t}(\hat{U})  \rangle
		\end{align}
for all $t\in[0, T]$ on a compact set of controls $u$. Indeed this is possible following an approach of \cite{Cuchiero23} who use stopped drivers in the context of proving a universal approximation theorem applicable to rough differential equations.
		\begin{defi}
			Given a function $h\colon[0,T] \to \mathbb{R}^m$ we define the stopped $h$ at time $s \in [0,T]$ as $h^s:[0,T] \to \mathbb{R}^m$ given by
			\begin{align*}
				h^s(t)\coloneqq h(t)\boldsymbol{1}_{[0, s]}(t),
			\end{align*} 
where $\boldsymbol{1}$ denotes the indicator function.
		\end{defi}
	Then, the strategy is to exploit that
	$x(s; u)=x(T;u^s)$ and to subsequently apply Theorem \ref{UAT} to $x(T;u^s)$ or the associated output. We formulate a theorem in the following that ensures \eqref{intu_UAT}. Let us also refer to \cite[Theorem 3.13]{Cuchiero23}, where such a result is established in the rough paths context.
	\begin{them}
	\label{UATUniform}
Suppose that $f_0$ and $f$ in  \eqref{ODE} are globally Lipschitz and that $y_i$ is the $i$th component of the quantity of interest $y(\cdot; u)=c(x(\cdot; u))$, $u\in L^2([0,T];\mathbb{R}^{m})$, with a locally Lipschitz continuous $c$ and $x=x(\cdot; u)$ being the solution of \eqref{ODE}. Further, let $\mathcal{K} \subset L^2([0,T];\mathbb{R}^{m})$ be compact. Then, for every $\varepsilon>0$ there exists an $\ell_i \in \mathcal W_{m+1}$ such that
		\begin{align*}
	\sup_{t \in [0,T]}	\sup_{u \in \mathcal{K}} |y_i(t; u)-\langle \ell_i, S_{0,t}(\hat{U})  \rangle|< \varepsilon.
		\end{align*} 
	\end{them}
	\begin{proof}
Let us first introduce the following auxiliary system: \begin{align*}
\dot{z}(t)=\mathfrak f(z(t))\mathfrak u(t),\quad t\in[0, T], \quad z(0)=x_0,                                                       \end{align*}
where $\mathfrak f= (f_0 , f)$ and $\mathfrak u\in L^1([0, T]; \mathbb R^{m+1})$.  Clearly, we have $z(t; \mathfrak u)=x(t; u)$ for all $t\in[0, T]$ if $\mathfrak u=\tilde u=(1, u^\top)^\top$.
 Let us define the stopped version of  $\tilde{u}$ by $\tilde u^s(t)\coloneqq \tilde u(t)\boldsymbol{1}_{[0, s]}(t)$ for $s\in [0, T]$. We observe that \begin{align*}
{z}(T; \tilde u^s)&=x_0+\int_0^T \mathfrak f(z(q; \tilde u^s))\tilde u^s(q) dq =    x_0+\int_0^T f_0(z(q; \tilde u^s)) \boldsymbol{1}_{[0, s]}(q) dq   + \int_0^T f(z(q; \tilde u^s))u(q)\boldsymbol{1}_{[0, s]}(q) dq\\
&=x_0+\int_0^s f_0(x(q; u)) dq   + \int_0^s f(x(q; u))u(q)dq = x(s; u)
\end{align*}
exploiting that $\tilde u\equiv \tilde u^s$ on $[0, s]$. Since $f_0$ and $f$ are globally Lipschitz continuous, the same holds for $\mathfrak{f}$. Therefore, we can apply Theorem \ref{Lemma1} yielding that $\mathfrak u\mapsto z(\cdot; \mathfrak{u})$ is continuous on $L^1([0,T];\mathbb{R}^{m+1})$. Given the locally Lipschitz map $c(\cdot)$ and using Remark \ref{rem1}, we obtain the continuity of $\mathfrak  u\mapsto \bar y(\cdot; \mathfrak u):= c(z(\cdot; \mathfrak{u}))$ on $L^1([0,T];\mathbb{R}^{m+1})$. The goal is to apply Theorem  \ref{UAT} to the continuous functional $\mathfrak u\mapsto\bar y_i(T; \mathfrak u)$, where the index $i$ indicates the $i$th component of a vector. In this context, let us now introduce the following set of stopped controls:
		\begin{align*}
			H\coloneqq \{ \tilde{u}^t : \, t \in [0,T], \tilde u=(1, u^\top)^\top, u \in \mathcal{K} \}.
		\end{align*}

		%We construct an auxiliary set of stopped functions extended by an other $1$ component $\tilde{u}(t) \mapsto (1,u(t)^\top)^\top$via the set 
		%\begin{align*}
		%	\tilde{H}\coloneqq  \widetilde{\{\tilde{u}^t}  : \, t \in [0,T], u \in \mathcal{K} \}
		%\end{align*}
		%where $\tilde{H} \subset L^p([0,T], \mathbb{R}^{m+1})$. \\

\textbf{Step 1:}
		 We prove that $H$ is relatively compact in $L^1([0,T]; \mathbb R^{m+1})$.
		From Theorem 1 in \cite{Hanche} we know that $\mathcal{F}\subset L^q([0,T];\mathbb{R}^m)$, $q\geq 1$,  is relatively compact if and only if
		\begin{align}\label{Hanche}
			\sup_{f\in \mathcal F} \int_0^T\|f(t+h)-f(t)\|^q {\mathrm d} t \rightarrow \,0
		\end{align}
as $h\to 0$. Notice that we set $f(t)=0$ for $t\not\in [0, T]$ in \eqref{Hanche}.  Since $\mathcal{K}$ is a compact set in $L^2([0,T];\mathbb{R}^m)$, we obtain from \eqref{Hanche} that
			\begin{align}\label{ass_lp_com}
	\sup_{u\in \mathcal K}\int_0^T \|u(s+h)-u(s)\|^q \, \mathrm{d}s \rightarrow \,0
		\end{align}
as  $h\to 0$ for $q=1, 2$. Suppose that $h>0$. Now, we consider \eqref{Hanche} for ${H}$ which yields
\begin{align*}
&\sup_{\mathfrak{u}\in  H}\int_0^T \|\mathfrak u(s+h)-\mathfrak u(s)\| \, \mathrm{d}s=
\sup_{t\in [0, T]}\sup_{u\in \mathcal K}\int_0^T \|\tilde{u}^t(s+h)-\tilde{u}^t(s)\| \, \mathrm{d}s\\
&\leq \sup_{t\in [0, T]}\int_0^T \vert\boldsymbol{1}_{[0, t]}(s+h)-\boldsymbol{1}_{[0, t]}(s)\vert \, \mathrm{d}s+\sup_{t\in [0, T]}\sup_{u\in \mathcal K}
\int_0^T \|u(s+h)\boldsymbol{1}_{[0, t]}(s+h)-u(s)\boldsymbol{1}_{[0, t]}(s)\| \, \mathrm{d}s.
\end{align*}
Firstly, we have that $\sup_{t\in [0, T]}\int_0^T \vert\boldsymbol{1}_{[0, t]}(s+h)-\boldsymbol{1}_{[0, t]}(s)\vert \, \mathrm{d}s= \sup_{t\in [0, T]}\int_{t-h}^t \, \mathrm{d}s=h\to 0$ for $h\to 0$. Secondly,  we obtain \begin{align*}
\int_0^T \|u(s+h)\boldsymbol{1}_{[0, t]}(s+h)-u(s)\boldsymbol{1}_{[0, t]}(s)\| \, \mathrm{d}s= \int_0^{t-h} \|u(s+h)-u(s)\|\, \mathrm{d}s+\int_{t-h}^t \|u(s)\| \, \mathrm{d}s.
\end{align*}
Condition \eqref{ass_lp_com} gives us \begin{align*}\sup_{t\in [0, T]}\sup_{u\in \mathcal K}\int_0^{t-h} \|u(s+h)-u(s)\| \, \mathrm{d}s\leq \sup_{u\in \mathcal K} \int_0^{T} \|u(s+h)-u(s)\| \, \mathrm{d}s\to 0
	\end{align*}
 as $h\to 0$. Furthermore, the Cauchy-Schwarz inequality leads to \begin{align*}
\sup_{t\in [0, T]}\sup_{u\in \mathcal K}\int_{t-h}^t \|u(s)\| \, \mathrm{d}s\leq \sqrt{h} \sup_{u\in \mathcal K} \sqrt{\int_{0}^T \|u(s)\|^2 \, \mathrm{d}s}\to 0
		\end{align*}
as $h\to 0$ exploiting that $\mathcal K$ is compact and hence bounded in $L^2([0, T]; \mathbb R^m)$. We omit the case of $h<0$, because it is analogue to considering $h>0$. Consequently, $H$ is relatively compact such that the closure  $\cl(H)$ is compact in $L^1([0, T]; \mathbb R^{m+1})$.
% \\
% (Vorschlag $L^1$ case)\\
% Since $\mathcal{K}$ is compact for all $\varepsilon>0$ there exists a finite subset $\mathcal{K}_\varepsilon \subset \mathcal{K}$ such that $\|u-u_\varepsilon\|<\varepsilon$ with $u_{\varepsilon} \in \mathcal{K}_\varepsilon$ for all $u\in \mathcal{K}$. It follows
% \begin{align*}
% \int_{t-y}^{t} \|u(s)\| \, \mathrm{d}s \leq\int_{t-y}^{t} \|u(s)-u_\varepsilon(s)\| \, \mathrm{d}s + \int_{t-y}^{t} \|u_\varepsilon(s)\| \, \mathrm{d}s \leq \varepsilon + \int_{t-y}^{t} \|u_\varepsilon(s)\| \, \mathrm{d}s
% \end{align*}
% and therefore
% \begin{align*}
% \sup_{{u} \in \mathcal{K}} \int_{t-y}^{t} \|u(s)\| \, \mathrm{d}s \leq \varepsilon +\sup_{{u} \in \mathcal{K}_\varepsilon} \int_{t-y}^{t} \|u(s)\| \, \mathrm{d}s \to 0
% \end{align*}
% with $y \to 0$ since $\mathcal{K}_\varepsilon$ is finite and by absolute continuity of Lebesgue integral.
% (Vorschlag Ende)\\
% \begin{align*}
% 	\int_0^T \|\tilde{u}^t(s+y)-\tilde{u}^t(s)\|^p \, \mathrm{d}s&=\int_0^T \|u(s+y)\boldsymbol{1}_{\{ s <t-y\}}-u(s)\boldsymbol{1}_{\{ s <t\}}\|^p \, \mathrm{d}s\\
% 	&=\int_0^{t-y} \|u(s+y)-u(s)\|^p \, \mathrm{d}s+\int_{t-y}^t \|u(s)\|^p \, \mathrm{d}s\\
% 	&\leq \varepsilon^p+\sup_{t\in[y,T]}\|u\|_{L^p([t-y,t],\mathbb{R}^m)} \leq \tilde{\varepsilon}^p
% 		\end{align*}
% 		which shows that $\cl({H})$, the closure of ${H}$, is compact.
		
\textbf{Step 2:}
We apply Theorem \ref{UAT} to $\mathfrak u\mapsto\bar y_i(T; \mathfrak u)$ on the compact set $\cl(H)$. This theorem guarantees the existence of a $\bar\ell_i \in \mathcal{W}_{m+2}$, such that
\begin{align*}
	&\sup_{t \in [0,T]}	\sup_{u \in \mathcal{K}} \Big\vert y_i(t; u)-\Big\langle \bar\ell_i, S_{0,T}\Big({\widehat{\int_{0}^{\cdot}\tilde{u}^t(s) \, \mathrm{d}s}}\Big)  \Big\rangle\Big\vert = \sup_{t \in [0,T]}	\sup_{u \in \mathcal{K}}\Big\vert\bar y_i(T; \tilde u^t)-\Big\langle \bar \ell_i, S_{0,T}\Big({\widehat{\int_{0}^{\cdot}\tilde{u}^t(s) \, \mathrm{d}s}}\Big)  \Big\rangle\Big\vert\\
	&= \sup_{\mathfrak u \in H}|\bar y_i(T; \mathfrak u)-\langle \bar \ell_i, S_{0,T}(\hat{\mathfrak U})  \rangle|\leq  \sup_{\mathfrak u \in \cl(H)}|\bar y_i(T; \mathfrak u)-\langle \bar \ell_i, S_{0,T}(\hat{\mathfrak U})  \rangle|<\varepsilon
\end{align*}
for an arbitrary $\varepsilon>0$, where $\hat{\mathfrak U}(t)=(t, \int_0^t\mathfrak u(s)^\top \mathrm{d}s)^\top$, $t\in [0, T]$, for $\mathfrak u\in L^1([0, T]; \mathbb R^{m+1})$.
% \begin{align*}
% \sup_{t \in [0,T]}	\sup_{u \in \mathcal{K}} |{{g}}{({u}^t})-\langle \tilde{\ell}, S_{0,T}({\widehat{\int_{0}^{\cdot}\tilde{u}^t(s) \, \mathrm{d}s}})  \rangle|= \sup_{{v} \in {H}} {g}({v})- \langle\tilde{\ell},  S_{0,T}({\hat{V}})   \rangle|\leq \sup_{{v} \in cl({H})} |{g}({v})- \langle\tilde{\ell},  S_{0,T}({\hat{V}})   \rangle|  \leq \varepsilon,
% 		\end{align*}
% 		where $V(\cdot)=\int_{0}^{\cdot} v(s) \, \mathrm{d}s$.\\

\textbf{Step 3:}
		We continue to show that for each  $\bar{\ell} \in \mathcal W_{m+2}$ there exists an $\ell \in \mathcal W_{m+1}$, such that for all $t \in [0,T]$ and $u \in \mathcal{K}$
		\begin{align*}
		\Big\langle\bar{\ell},S_{0,T} \Big({\widehat{\int_{0}^{\cdot}\tilde{u}^t(s) \, \mathrm{d}s}}\Big) \Big\rangle=\big\langle\ell, S_{0,t}\big(\hat{U}\big)\big\rangle
		\end{align*} 
holds, where each  ``hat'' indicates that a time component is added to the path. Since we can always find an $N\in\mathbb N$ and a vector $\bar l\in \mathbb R^{{\bar n}_N}$ with $\langle\bar{\ell},S_{0,T} ({\widehat{\int_{0}^{\cdot}\tilde{u}^t(s) \, \mathrm{d}s}}) \rangle=\langle\bar{l},S_{0,T}^N ({\widehat{\int_{0}^{\cdot}\tilde{u}^t(s) \, \mathrm{d}s}}) \rangle$, it is enough to show that \begin{align*}
		\Big\langle\bar{l},S_{0,T}^N \Big({\widehat{\int_{0}^{\cdot}\tilde{u}^t(s) \, \mathrm{d}s}}\Big) \Big\rangle=\big\langle l, S_{0,t}^N\big(\hat{U}\big)\big\rangle
		\end{align*}
for some $l\in\mathbb R^{n_N}$.
Notice that the dimensions $n_N$ and $\bar n_N$ depend on the degree $N$ of the truncated signature and will not be specified further. Let $\mathfrak{u}\in L^2([0, T]; \mathbb R^{m+2})$ %=(1, 1^t, (u^t)^\top)^{\top}$
and $x=x(t;\mathfrak{u})$ be the solution to
\begin{align*}
	\dot{x}(t)= \sum_{i=-1}^{m} A_i x(t) \mathfrak{u}_i, \quad x(0)=\mathbf{e}_1\in \mathbb R^{{\bar n}_N},
\end{align*}
where the matrices $A_i$ are like in Proposition \ref{SigODE}. Then, we have $x(v)=x(v;(1, 1^t, (u^t)^\top)^{\top})=S^N_{0,v}(\widehat{\int_{0}^{\cdot} \tilde{u}^t(s) \, \mathrm{d}s)}$ and $z(v)=x(v;(1, 1, u^\top)^{\top})=S^N_{0,v}\Big(\widehat{\widehat{U}}\Big)=S^N_{0,v}(\widehat{\int_{0}^{\cdot} \tilde{u}(s) \, \mathrm{d}s)}$.
Suppose that $v>t$. Then, exploiting that $x$ and $z$ coincide on $[0, t]$, we have
%{\allowdisplaybreaks
\begin{align*}
	&S^N_{0,v}\Big(\widehat{\int_{0}^{\cdot} \tilde{u}^t(s) \, \mathrm{d}s\Big)}=x(v)\\
	&=\mathbf{e}_1+ \int_{0}^{v} A_{-1} x(s)  \, \mathrm{d}s + \int_{0}^{v} A_{0} x(s) \boldsymbol{1}_{[0,t]}(s) \, \mathrm{d}s + \sum_{i=1}^m \int_{0}^{v} A_i x(s) u(s) \boldsymbol{1}_{[0,t]}(s) \, \mathrm{d}s\\
	&= \mathbf{e}_1+ \int_{0}^{t} A_{-1} z(s)  \, \mathrm{d}s + \int_{0}^{t} A_{0} z(s) \, \mathrm{d}s + \sum_{i=1}^m \int_{0}^{t} A_i z(s) u(s) \, \mathrm{d}s +\int_{t}^{v} A_{-1} x(s)  \, \mathrm{d}s\\
	&=S^N_{0,t}\Big(\widehat{\widehat{U}}\Big)+\int_{t}^{v} A_{-1} S^N_{0,s}\Big(\widehat{\int_{0}^{\cdot} \tilde{u}^t(s) \, \mathrm{d}s\Big)}  \, \mathrm{d}s.
\end{align*}
%}
		Now, we observe that
		\begin{align*}
			A_{-1}S^N_{0,s}\Big(\widehat{\int_{0}^{\cdot} \tilde{u}^t(s) \, \mathrm{d}s\Big)}=
			\tilde A_{-1} S^{N-1}_{0,s}\Big(\widehat{\int_{0}^{\cdot} \tilde{u}^t(s) \, \mathrm{d}s\Big)},
		\end{align*}
where $\tilde{A}_{-1}=\begin{bmatrix}
			\mathcal{O}_{1,\tilde{n}}\\
			\blkdiag(e_1, \ldots ,e_1) 
		\end{bmatrix}$, $\tilde n={\bar n}_{N-1}$.
Given $\bar{l}^N\in\mathbb R^{{\bar n}_N}$, we therefore obtain that
		\begin{align}\label{recursion}
			\begin{split}
		\Big\langle \bar{l}^N,  S^N_{0,v}\Big(\widehat{\int_{0}^{\cdot} \tilde{u}^t(s) \, \mathrm{d}s}\Big)\Big\rangle = \Big\langle \bar{l}^N,  S^N_{0,t}\Big(\widehat{\widehat{U}}\Big)\Big\rangle+ \int_{t}^{v} \Big\langle \tilde{A}_{-1}^\top \bar{l}^N, S_{0,s}^{N-1}\Big(\widehat{\int_{0}^{\cdot} \tilde{u}^t(s) \, \mathrm{d}s}\Big) \Big\rangle \, \mathrm{d}s \\
		= \big\langle {l}^N,  S^N_{0,t}\big({\hat{U}}\big)\big\rangle+ \int_{t}^{v} \Big\langle \bar{l}^{N-1}, S_{0,s}^{N-1}\Big(\widehat{\int_{0}^{\cdot} \tilde{u}^t(s) \, \mathrm{d}s}\Big) \Big\rangle \, \mathrm{d}s,
		\end{split}
		\end{align}
where $\bar{l}^{N-1}\coloneqq \tilde{A}_{-1}^\top\bar{l}^{N}\in\mathbb R^{{\bar n}_{N-1}}$ and ${l}^N\in\mathbb R^{n_N}$. With $\bar{l}\in\mathbb R^{\bar n_N}$, we apply \eqref{recursion} to the following functional twice and get
		\begin{align*}
			\Big\langle \bar{l},  S^N_{0,T}\Big(\widehat{\int_{0}^{\cdot} \tilde{u}^t(s) \, \mathrm{d}s\Big)}\Big\rangle&=\big\langle {l}^N,  S^N_{0,t}\big({\hat{U}}\big)\big\rangle + \int_{t}^{T} \Big\langle \bar {l}^{N-1},  S^{N-1}_{0,s_{N-1}}\Big(\widehat{\int_{0}^{\cdot} \tilde{u}^t(s) \, \mathrm{d}s}\Big)\Big\rangle \mathrm{d}s_{N-1}\\
			&=\big\langle {l}^N,  S^N_{0,t}\big({\hat{U}}\big)\big\rangle + \int_{t}^{T} \big\langle {\tilde{l}}^{N-1},  S^{N-1}_{0,t}\big({\hat{U}}\big)\big\rangle \mathrm{d}s_{N-1} \\ &\quad+ \int_{t}^T \int_t^{s_{N-1}} \Big\langle \bar l^{N-2}, S^{N-2}_{0,s_{N-2}}\Big(\widehat{\int_{0}^{\cdot} \tilde{u}^t(s) \, \mathrm{d}s}\Big) \Big\rangle \, \mathrm{d} s_{N-2}  \, \mathrm{d} s_{N-1}\\
				&=\big\langle {{l}}^N,  S^N_{0,t}\big({\hat{U}}\big)\big\rangle + (T-t) \big\langle \tilde {l}^{N-1},  S^{N-1}_{0,t}\big({\hat{U}}\big)\big\rangle \\ &\quad+ \int_{t}^{T} \int_t^{s_{N-1}} \Big\langle \bar{l}^{N-2} , S^{N-2}_{0,s_{N-2}}\Big(\widehat{\int_{0}^{\cdot} \tilde{u}^t(s) \, \mathrm{d}s}\Big) \Big\rangle \, \mathrm{d} s_{N-2}  \, \mathrm{d} s_{N-1}\\
				&=\big\langle {{l}}^N,  S^N_{0,t}\big({\hat{U}}\big)\big\rangle+\big\langle {{l}}^{N-1},  S^N_{0,t}\big({\hat{U}}\big)\big\rangle\\
				&\quad+\int_{t}^{T} \int_t^{s_{N-1}} \Big\langle \bar{l}^{N-2}, S^{N-2}_{0,s_{N-2}}\Big(\widehat{\int_{0}^{\cdot} \tilde{u}^t(s) \, \mathrm{d}s}\Big)\Big\rangle \, \mathrm{d} s_{N-2} \, \mathrm{d} s_{N-1},
			\end{align*}
where ${l}^{N-1}, {l}^{N}\in\mathbb R^{n_N}$, $\tilde {l}^{N-1}\in\mathbb R^{n_{N-1}}$ and $\bar{l}^{N-2}\in\mathbb R^{{\bar n}_{N-2}}$.
A successive application of \eqref{recursion} yields 
\begin{align*}
	\Big\langle \bar{l}, S^N_{0,T}\Big(\widehat{\int_{0}^{\cdot} \tilde{u}^t(s) \, \mathrm{d}s}\Big)\Big\rangle&=\sum_{i=0}^{N-1} \big\langle l^{N-i}, S_{0,t}^N\big(\hat{U}\big) \big\rangle\\
	&\quad+ \int_{t}^{T}\int_{t}^{s_{N-1}} \cdots \int_{t}^{s_1} \Big\langle \bar{l}^0, S_{0,s_0}^0\Big(\widehat{\int_{0}^{\cdot} \tilde{u}^t(s) \, \mathrm{d}s}\Big) \Big\rangle\, \mathrm{d}s_0 \cdots \mathrm{d}s_{N-1} \\
	&=\sum_{i=0}^{N-1} \big\langle l^{N-i}, S_{0,t}^N\big(\hat{U}\big) \big\rangle +  \bar{l}^0 \frac{(T-t)^N}{N!}=\sum_{i=0}^{N} \big\langle l^{N-i}, S_{0,t}^N\big(\hat{U}\big) \big\rangle \\
	&= \langle {l}, S_{0,t}^N(\hat{U}) \rangle,
\end{align*}
where $S^0\equiv 1$, $\bar l^0\in\mathbb R$, $l^i\in\mathbb R^{n_N}$ ($i\in\{0, \dots, N\}$) and $l \coloneqq\sum_{i=0}^N l^{N-i}\in\mathbb R^{n_N}$. This concludes the proof.
	\end{proof}
The next corollary shows that we do not need to assume global Lipschitz continuity of the vector fields, but only a one-sided linear growth condition. This is vital if polynomial nonlinearities like in Example  \ref{ex1} shall be covered.
\begin{kor}\label{cor1}
If we consider locally Lipschitz continuous vector fields $f_0$ and $f$ together with the one-sided linear growth condition\begin{align*}
     2\langle x, f_0(x)\rangle + \big\|f(x)\big\|^2_F\leq \mathcal C(1+\|x\|^2)                                                                                                                           \end{align*}
for some constant $\mathcal C$ and all $x\in\mathbb R^d$ instead of the global Lipschitz assumption on $f_0$ and $f$ in Theorem  \ref{UATUniform}, the statement of that theorem remains true, i.e., for every $\varepsilon>0$ there exists an $\ell_i \in \mathcal W_{m+1}$ such that
		\begin{align*}
	\sup_{t \in [0,T]}	\sup_{u \in \mathcal{K}} |y_i(t; u)-\langle \ell_i, S_{0,t}(\hat{U})  \rangle|< \varepsilon.
		\end{align*}
\end{kor}
\begin{proof}
Since $\mathcal K$ is compact, it is also bounded. This means that we find an $R>0$, such that $\|u\|_{L^2([0, T])}\leq R$ for all $u\in \mathcal K$. We can now use the proof of Lemma \ref{Hilfslemma} case (b), where it is shown that $x(\cdot, u)$ takes values in a ball $B$ with finite radius for $u\in \mathcal K$, see \eqref{x_in_ball2}.
As $f_0$ and $f$ are locally Lipschitz, they are Lipschitz continuous on $B$. Therefore, one can find globally Lipschitz continuous extensions $\bar f_0$ and $\bar f$ on $\mathbb R^d$, so that $\bar f_0\big\vert_B=f_0$ and $\bar f\big\vert_B=f$, see for instance \cite{lip_extend}. We introduce the following modified version of \eqref{ODE}
 \begin{align}\label{mod_ode}
		\dot{\bar x}(t)=\bar f_0(\bar x(t))+\bar f(\bar x(t))u(t),\quad \bar x(0)=x_0,
	\end{align}
having globally Lipschitz vector fields. Now, Theorem \ref{UATUniform} can be applied to \eqref{mod_ode}. The result follows, because $\bar x(t; u)= x(t; u)$ for $u\in \mathcal K$.
\end{proof}
For completeness, let us formulate a corollary, where the whole output $y$ is approximated by a linear map of the truncated signature of $\hat U$.
\begin{kor}\label{cor_approx_y}
Suppose that the assumptions of Corollary \ref{cor1} hold. Moreover, let $y$ be the quantity of interest introduced in \eqref{eq_y}. Then, for every $\varepsilon>0$ there exists an $N\in\mathbb N$ and a matrix $C\in\mathbb R^{p\times n}$ with $n=n(N, m+1)$, such that
		\begin{align*}
	\sup_{t \in [0,T]}	\sup_{u \in \mathcal{K}} \|y(t; u)-C S^N_{0,t}(\hat{U}) \|< \varepsilon.
		\end{align*}
\end{kor}
		\begin{proof}
 Directly applying Corollary \ref{cor1}, we find an $\ell_i\in\mathcal W_{m+1}$, such that \begin{align*}
	\sup_{t \in [0,T]}	\sup_{u \in \mathcal{K}} |y_i(t; u)-\langle \ell_i, S_{0,t}(\hat{U})  \rangle|< \frac{\varepsilon}{p}
\end{align*}
for a given $\varepsilon>0$. As $\ell_i$ has only finitely many entries different from zero for all $i=1, \dots, p$, we find an $N\in\mathbb N$ and vectors $l_i\in \mathbb R^n$, so that $\langle \ell_i, S_{0,t}(\hat{U})  \rangle=\langle l_i, S_{0,t}^N(\hat{U})\rangle$ for all $i=1, \dots, p$. Defining $C$ as a matrix with $i$th row equal to $l_i^\top$, we find that \begin{align*}
	\sup_{t \in [0,T]}	\sup_{u \in \mathcal{K}} \|y(t; u)-C S^N_{0,t}(\hat{U}) \|\leq\sum_{i=1}^p
	\sup_{t \in [0,T]}	\sup_{u \in \mathcal{K}} |y_i(t; u)-\langle l_i, S_{0,t}^N(\hat{U})\rangle|
	< \varepsilon.
\end{align*}
This concludes the proof.
\end{proof}
An immediate consequence of Corollary \ref{cor_approx_y} and Proposition \ref{SigODE} is that we obtain a universal bilinear system \eqref{tr_sig_ode} satisfied by $S^N(\hat{U})$. The matrices $A_0, A_1, \dots, A_m$ are known according to Proposition \ref{SigODE}. Only the matrix $C$ needs to be computed. This does not require any knowledge on \eqref{original_system}, but only output data, an advantage for practical considerations. We will discuss this aspect further in Section \ref{sec_learn_C}. The resulting bilinearization that we found has several additional advantages in comparison to existing methods like the Carleman (bi)linearization \cite{Rugh1981, Sastry1999}. The dimension of the signature system \eqref{tr_sig_ode} grows only in the number of inputs $m$, see Proposition \ref{SigODE}. Thus, our method is applicable when the state variable is high-dimensional (but the number of inputs is low). However, notice that the dimension $n=\frac{(m+1)^{N+1}-1}{m}$ of the truncated signature $S^N(\hat{U})$ grows exponentially in $m$. Therefore, \eqref{tr_sig_ode} is expected to be high-dimensional. For that reason, we consider MOR for such bilinear systems in Section \ref{sec_mor}. In addition, let us point out that we do not need any smoothness of the vector fields in our approximation, see the assumptions of Corollary \ref{cor1}, making the signature approach widely applicable.

% We conclude this section by illustrating that Corollary \ref{UATUniform} is a generalization of the Weierstrass approximation theorem to Stieltjes-differential equations \eqref{Stieltjes-DE}.
% 	\begin{bsp}
% Consider the Stieltjes differential equation \eqref{Stieltjes-DE} with $U\equiv 0$, which means we are in the uncontrolled ODE case. Then, we have $\hat U(t)=t$, where its signature \begin{align*}
% 				S_{0,t}(\hat U)=\Big(1,\, t ,\, \frac{t^2}{2!},\,\frac{t^3}{3!}, \, \ldots\Big)
% 			\end{align*}
% is the vector of Taylor polynomials on $[0, T]$. As $t\mapsto y(t)$ is continuous, we have
% 						\begin{align*}
% 				\sup_{t \in [0,T]} |y(t)-CS_{0,t}^N(\hat U(t))|< \varepsilon
% 			\end{align*}
% for some sufficiently large $N$ and a matrix $C$ being a consequence of the classical Weierstrass approximation theorem in this simple case.
% 	\end{bsp}

	\section{MOR for unstable bilinear systems with non-zero initial states}\label{sec_mor}

We briefly discuss a strategy on how to reduce the dimension of system \eqref{tr_sig_ode}. In particular, there will be no assumptions on the matrices $A_0, A_1, \dots, A_m$ such as stability. This is natural as the truncated signature $S^N(\hat U)$ satisfies such an equation with coefficients having only zero eigenvalues, see Proposition \ref{SigODE} and Remark \ref{remark_sig_bil}. In addition, we assume the initial state to be different from zero as this is true for the truncated signature process as well. We even consider a subspace of initial states in the following meaning that we have $s_0= S_0 v$, where the columns of the matrix $S_0$ span this space of initial values, whereas the generic vector $v$ can be viewed as another input to the system. If a single initial value $s_0$ in \eqref{tr_sig_ode} is considered, we simply set $S_0=s_0$ and $v=1$. It is important to notice that there are first results on MOR for bilinear systems with non-zero initial states \cite{Cao2021, Redmann2023}. However, the Gramians considered in these works require a certain stability of the system. This is not guaranteed here. For that reason, we extend the approach of \cite{Redmann2023} to unstable systems using Gramians on finite intervals $[0, T]$.

\subsection{Identifying unimportant directions}\label{sec_gram}
We need the fundamental solution of the truncated signature equation \eqref{tr_sig_ode} in the analysis of the Gramians proposed below. These Gramians are crucial for detecting dominant subspaces of \eqref{tr_sig_ode} and are hence the basis for computing the reduced system \eqref{rom_intro}.
	\begin{defi}\label{defn_fund}
Suppose that $t_0\leq t\leq T$. The fundamental solution of \eqref{tr_sig_ode} is the function $\Phi$ that satisfies
 \begin{align*}
 \Phi(t,t_0) = I +\int_{t_0}^t A_0 \Phi(v,t_0) \mathrm{d}v + \sum_{i=1}^m \int_{t_0}^t A_i \Phi(v,t_0) u_i(v) \mathrm{d}v,
\end{align*}
where $I$ is the identity matrix. If $t_0=0$, we set $\Phi(t):=\Phi(t, 0)$, $t\in [0, T]$.
\end{defi}
By definition of $\Phi$, we immediately obtain a solution representation $S(t)= \Phi(t, t_0) S(t_0)$, $t\in [t_0, T]$, for the solution of  \eqref{tr_sig_ode} which we exploit below. In the following, we find an upper bound for a quadratic form of $\Phi$ that is used to define Gramians.
\begin{lem}\label{fund_est}
Let $\Phi$ be the fundamental solution according to Definition \ref{defn_fund} and $M$ a positive semidefinite matrix. Then, it holds that
 \begin{align*}
\Phi(t, t_0) M \Phi(t, t_0)^\top \leq  \exp\left\{\int_{t_0}^t \left\|u(v)\right\|^2 \mathrm{d}v\right\}  Z(t-t_0),
\end{align*}
where $ Z(t)$, $t\in [0, T]$, satisfies the matrix differential equation
\begin{align}\label{eqZ}
 \dot {Z}(t) = A_0  Z(t) +  Z(t) A_0^\top +\sum_{i=1}^m A_i Z(t) A_i^\top ,\quad Z(0) = M.
\end{align}
\end{lem}
\begin{proof}
A proof is stated in \cite{Redmann2021} and in \cite{Redmann2023}.
\end{proof}
Below, let us denote the solution of \eqref{eqZ} with initial state $M$ by $Z(t, M)$, $t\in[0, T]$. Consequently, we can define a reachability Gramian by
\begin{align*}
  P:=\int_0^T Z(t, S_0 S_0^\top) \mathrm{d}t.                                                                                                                                                                                                                                                                                           \end{align*}
The dependence of $P$ on $T$ is omitted for a simpler notation. The following proposition creates a link between the eigenspaces of the reachability Gramian $P$ and state directions that can be neglected in \eqref{tr_sig_ode}.
\begin{prop}\label{unimportant_state0}
Let $S(t)$, $t\in [0, T]$, be the solution of \eqref{tr_sig_ode} and $(p_i)$ be an orthonormal basis of $\mathbb R^n$ consisting of eigenvectors of the reachability Gramian $P$. Then, we have
 \begin{align*}
 \int_0^T \vert\langle S(t), p_{i}\rangle \vert^2 \mathrm{d}t\leq \lambda_{i}\exp\left\{\left\| u\right\|_{L^2([0, T])}^2\right\} \left\|v\right\|^2,
                    \end{align*}
where $\lambda_{i}$ is the eigenvalue corresponding to $p_{i}$.
\end{prop}
\begin{proof}
Below, let us exploit that
  \begin{align*}
 S(t) = \Phi(t, 0) S(0)=\Phi(t) s_0 = \Phi(t) S_0 v,\quad t\in [0, T].
\end{align*}
Using the inequality of Cauchy-Schwarz yields\begin{align*}
  \int_0^T \langle S(t), p_{i}\rangle^2 \mathrm{d}t =  \int_0^T \langle \Phi(t) S_0 v, p_{i}\rangle^2 \mathrm{d}t = \int_0^T \langle  v, S_0^\top \Phi(t)^\top p_{i}\rangle^2 \mathrm{d}t  \leq \left\|v\right\|^2 p_{i}^\top  \int_0^T \Phi(t) S_0 S_0^\top \Phi(t)^\top \mathrm{d}t\, p_{i}.
  \end{align*}
Based on Lemma \ref{fund_est}, we obtain\begin{align*}
  \int_0^T \langle S(t), p_{i}\rangle^2 \mathrm{d}t
   &\leq \left\|v\right\|^2 \exp\left\{\int_0^T \left\|u(t)\right\|^2 \mathrm{d}t\right\}  p_{i}^\top  \int_0^T Z(t, S_0 S_0^\top) \mathrm{d}t\, p_{i} \\
   &\leq \left\|v\right\|^2\exp\left\{\left\|u\right\|_{L^2([0, T])}^2\right\} p_{i}^\top  P\, p_{i}
   =\left\|v\right\|^2\exp\left\{\left\|u\right\|_{L^2([0, T])}^2\right\}  \lambda_{i}.
                       \end{align*}
This concludes the proof.
                       \end{proof}
Proposition \ref{unimportant_state0} implies that we can neglect the eigenvector $p_i$ of $P$ in the bilinear model if the corresponding eigenvalue $\lambda_{i}$ is small. Therefore, eigenspaces of $P$ associated with small eigenvalues can be removed from the system to obtain a reduced model. Let us now continue by identifying redundancies in the linear map $CS(t)$, $t\in [0, T]$, characterized by the matrix $C$.
Let us consider $Z^*= Z^*(t, C^\top C)$, $t\in [0, T]$, satisfying \begin{align}\label{eqZad}
 \dot {Z}^*(t) = A_0^\top Z^*(t) +  Z^*(t) A_0+\sum_{i=1}^m A_i^\top Z^*(t) A_i,\quad Z^*(0) = C^\top C,
\end{align}
where the superscript $*$ indicates that the operator defining the right side of \eqref{eqZad} is the adjoint operator of the one occurring in \eqref{eqZ}. Based on the solution of \eqref{eqZad}, we define an observability Gramian \begin{align*}
Q:=\int_0^T Z^*(t, C^\top C) \mathrm{d}t                                                                                                                                                                                                \end{align*}
again omitting the dependence on the terminal time $T$. Let $(q_i)$ now be an orthonormal basis of $\mathbb R^n$ consisting of eigenvectors of $Q$. Then, we can always represent $S(t_0)=\sum_{i=1}^n \langle S(t_0), q_{i}\rangle q_{i}$ for any $t_0\in[0, T]$. Since we can write the output of the signature equation as \begin{align*}
 y_S(t)= C S(t) = C \Phi(t, t_0) S(t_0)= \sum_{i=1}^n \langle S(t_0), q_{i}\rangle  C \Phi(t, t_0) q_i,     \quad t\in[t_0, T],                                                                                                                                                                                                                                                                                                                \end{align*}
we are now able to analyze the contribution of the direction $q_i$ to $y_S(t)$ on the interval  $[t_0, T]$. In fact, we need to measure the magnitude of $C \Phi(t, t_0) q_i$ on this interval. This result is stated in the next proposition.
\begin{prop}\label{redundant_out}
Let $(q_{i})$ be an orthonormal basis of $\mathbb R^n$ consisting of eigenvectors of the observability Gramian $Q$. Then, we have
 \begin{align*}
  \int_{t_0}^T \left\|C \Phi(t, t_0) q_{i}\right\|^2 \mathrm{d}t \leq \mu_{i}\exp\left\{\left\|u\right\|_{L^2([0, T])}^2\right\},
                    \end{align*}
where $0\leq t_0<T$ and $\mu_{i}$ is the eigenvalue associated with $q_{i}$.
\end{prop}
\begin{proof}
Applying Lemma \ref{fund_est}, we find that \begin{align}\nonumber
  \int_{t_0}^T \left\|C \Phi(t, t_0) q_{i}\right\|^2 \mathrm{d}t &=     \int_{t_0}^T q_{i}^\top  \Phi(t, t_0)^\top C^\top C \Phi(t, t_0) q_{i} \mathrm{d}t = \int_{t_0}^T \trace\left(C \Phi(t, t_0) q_{i}q_{i}^\top\Phi(t, t_0)^\top C^\top\right) \mathrm{d}t \\ \nonumber&\leq
  \int_{t_0}^T \trace\left(C \exp\left\{\int_{t_0}^t \left\|u(v)\right\|^2 \mathrm{d}v\right\}  Z(t-t_0, q_{i}q_{i}^\top) C^\top\right) \mathrm{d}t \\ \nonumber&\leq
  \exp\left\{\left\| u\right\|_{L^2([0, T])}^2\right\}\int_{0}^T \trace\left(C  Z(t, q_{i}q_{i}^\top) C^\top\right) \mathrm{d}t\\ \label{insert_here}
 &=  \exp\left\{\left\| u\right\|_{L^2([0, T])}^2\right\} \int_{0}^T \trace\left(C^\top C Z(t, q_{i}q_{i}^\top)\right)\mathrm{d}t.
                   \end{align}
We exploit the relation between the trace and the vectorization $\vect(\cdot)$ of the matrix yielding  $\trace\left(C^\top C Z(t, q_{i}q_{i}^\top)\right)=\langle\vect(C^\top C), \vect(Z(t, q_{i}q_{i}^\top)\rangle$. Applying the vectorization to both sides of \eqref{eqZ} with $M=q_{i}q_{i}^\top$, we obtain \begin{align}\label{vec_mat_ODE}
 \frac{\mathrm{d}}{\mathrm{d}t} {\vect(Z(t))} = {K} \vect(Z(t)),\quad \vect(Z(0)) = \vect(q_{i}q_{i}^\top),\quad t\in [0, T],
 \end{align}
where we define \begin{align}\label{def_kron_mat}
K=A_0 \otimes I + I \otimes A_0 +\sum_{i=1}^m A_i\otimes A_i.
                \end{align}
The solution of \eqref{vec_mat_ODE} is $\expn^{K t}  \vect(q_{i}q_{i}^\top)$. Therefore, we obtain \begin{align*}
  \trace\left(C^\top C Z(t, q_{i}q_{i}^\top)\right)&=\langle\vect(C^\top C), \vect(Z(t, q_{i}q_{i}^\top)\rangle =     \langle\vect(C^\top C), \expn^{K t}  \vect(q_{i}q_{i}^\top)\rangle\\
  &=\langle\expn^{K^\top t} \vect(C^\top C), \vect(q_{i}q_{i}^\top)\rangle.
\end{align*}
By vectorizing \eqref{eqZad}, we observe that $\vect({Z}^*(t, C^\top C))=\expn^{K^\top t} \vect(C^\top C)$. Hence, we have
\begin{align*}
  \trace\left(C^\top C Z(t, q_{i}q_{i}^\top)\right)= \langle\vect({Z}^*(t, C^\top C)), \vect(q_{i}q_{i}^\top)\rangle = \trace\left(q_{i}q_{i}^\top {Z}^*(t, C^\top C)\right).
 \end{align*}
Inserting this into \eqref{insert_here} yields
\begin{align*}
  \int_{t_0}^T \left\|C \Phi(t, t_0) q_{i}\right\|^2 \mathrm{d}t &\leq
   \exp\left\{\left\| u\right\|_{L^2([0, T])}^2\right\} \int_{0}^T  \trace\left(q_{i}q_{i}^\top {Z}^*(t, C^\top C)\right)\mathrm{d}t\\
  & = \exp\left\{\left\| u\right\|_{L^2([0, T])}^2\right\} q_{i}^\top \int_{0}^T  {Z}^*(t, C^\top C)\mathrm{d}t\, q_{i}= \exp\left\{\left\| u\right\|_{L^2([0, T])}^2\right\} \mu_{i}.
  \end{align*}
This concludes the proof.
\end{proof}
Proposition \ref{redundant_out} identifies the eigenspaces of $Q$ to be of low relevance for the output $y_S$ in the bilinear model \eqref{tr_sig_ode} if they are associated with small eigenvalues $\mu_{i}$. In summary, this means that the truncation of eigenspaces of $P$ and $Q$ corresponding to small eigenvalues will not affect the dynamics much. However, the eigenspaces of $P$ and $Q$ do not coincide. Therefore, a transformation is introduced in the following section that creates a system with equal and diagonal Gramians.

\subsection{Reduced order model by time-limited balanced truncation}\label{sec_rom}

Suppose that $\mathcal T\in\mathbb R^{n\times n}$ is an invertible matrix that is used to define the variable $S_b=\mathcal T S$ with $S$ being the solution of the ODE in \eqref{tr_sig_ode}. The matrix $\mathcal T$ is chosen to simultaneously diagonalize the Gramians $P$ and $Q$ of \eqref{tr_sig_ode} in case such a transformation exists. We insert $S=\mathcal T^{-1} S_b$ into \eqref{tr_sig_ode} yielding
\begin{equation}
\label{tr_sig_trans}
\begin{aligned}
		\dot{S_b}(t)&=\mathcal T A_0 \mathcal T^{-1} S_b(t) + \sum_{i=1}^{m}\mathcal T A_i\mathcal T^{-1} S_b(t)u_i(t),\quad S_b(0)=\mathcal T s_0=\mathcal T S_0 v,\\
		y_S(t)&={C\mathcal T^{-1}}S_b(t),
	\end{aligned}
	\end{equation}
preserving the output $y_S$ for a given input $u$. However, \eqref{tr_sig_trans} has modified Gramians. Let $Z_b(\cdot, \mathcal T S_0 S_0^\top \mathcal T^\top)$ and $Z_b^*(\cdot, \mathcal T^{-\top} C^\top C\mathcal T^{-1})$ denote the solutions of \eqref{eqZ} and \eqref{eqZad}, respectively, when replacing $(A_0, \ldots, A_m)$ by $\Big(\mathcal T A_0 \mathcal T^{-1},\ldots, \mathcal T A_m\mathcal T^{-1} \Big)$. Then, the reachability and observability Gramians of \eqref{tr_sig_trans} are
\begin{align}\label{def_bal_gram}
  P_b=\int_0^T Z_b(t, \mathcal T S_0 S_0^\top \mathcal T^\top) \mathrm{d}t,\quad Q_b=\int_0^T Z_b^*(t, \mathcal T^{-\top} C^\top C\mathcal T^{-1}) \mathrm{d}t.
  \end{align}
  If we multiply \eqref{eqZ} with $\mathcal T$ from the left and with $\mathcal T^\top$ from the right, we see that $Z_b(t, \mathcal T S_0 S_0^\top \mathcal T^\top)=  \mathcal T  Z(t,S_0 S_0^\top) \mathcal T^\top$. Multiplying \eqref{eqZad} with $\mathcal T^{-\top}$ from the left and with $\mathcal T^{-1}$ from the right, we also observe that $Z_b^*(\cdot, \mathcal T^{-\top} C^\top C\mathcal T^{-1})= \mathcal T^{-\top} Z^*(\cdot,  C^\top C)\mathcal T^{-1}$. Inserting these identities into \eqref{def_bal_gram}, we obtain that the Gramian of \eqref{tr_sig_trans} are given by
  \begin{align}\label{transformed_gram}                                                                     P_b= \mathcal  T P \mathcal T^\top, \quad
   Q_b = \mathcal T^{-\top} Q \mathcal T^{-1}.
 \end{align}
Now, we specify the balancing transformation $\mathcal T$ that simultaneously diagonalizes both Gramians.
\begin{prop}\label{prop_bal}
Suppose that the reachability Gramian $P$ and the observability Gramian $Q$ are positive definite. We achieve that  $P_b=Q_b = \Sigma= \diag(\sigma_1,\ldots,\sigma_n)$ using the balancing transformation \begin{equation}\label{bal_transform}
  \mathcal T=\Sigma^{\frac{1}{2}} V^\top L_P^{-1},
\end{equation}
with the factorization $P=L_PL_P^\top$ and the spectral decomposition $L_P^\top QL_P=V\Sigma^2 V^\top$, where $V$ is orthogonal. Moreover, $\sigma_1^2, \dots, \sigma_n^2$ are the eigenvalues of $PQ$.
\end{prop}
\begin{proof}
 We insert \eqref{bal_transform} into \eqref{transformed_gram} and obtain $P_b= \Sigma^{\frac{1}{2}} V^\top L_P^{-1} P L_P^{-\top} V \Sigma^{\frac{1}{2}}=\Sigma$ as well as $Q_b= \Sigma^{-\frac{1}{2}} V^\top L_P^{\top} Q L_P V \Sigma^{-\frac{1}{2}}=\Sigma$. Since the spectral decomposition of $L_P^\top QL_P$ is computed, we know $\sigma_1^2, \dots, \sigma_n^2$ are the eigenvalues of this matrix. On the other hand, $L_P^\top QL_P$ has the same spectrum like $L_PL_P^\top Q = PQ$. This leads to the last statement of this proposition concluding the proof.
\end{proof}
Using the balancing transformation in \eqref{bal_transform}, we know that the Gramians are both the same and diagonal. Consequently, their eigenvectors represent the canonical basis of $\mathbb R^n$ meaning that $p_i=q_i=e_i$ in Propositions \ref{unimportant_state0} and \ref{redundant_out}, where $e_i$ is the $i$th unit vector in $\mathbb R^n$. Looking from this perspective at these propositions, we can conclude that the $i$th component of the balanced state $S_b$ is of low relevance if the Hankel singular value $\sigma_i$ is small. Let us introduce the partition $S_b(t)=\begin{bmatrix}
S_1(t)\\ S_2(t)
\end{bmatrix}$, $t\in [0, T]$, of the balanced state assuming that entries of $S_1(t)\in\mathbb R^r$ belong to the Hankel singular values  $\sigma_1,\ldots,\sigma_{r}$. Moreover, the variables in $S_2(t)\in\mathbb R^{n-r}$ are supposed to be truncated as they correspond to the $n-r$ smallest values $\sigma_{r+1},\ldots,\sigma_n$. If these truncated values are small, a low error in the approximation is expected. The reduced order model involves coefficients that we obtain from the following partition
  \begin{equation*}
  \begin{aligned}
\mathcal T A_i \mathcal T^\top  = \begin{bmatrix}{\widecheck A_i}&\star\\
\star&\star\end{bmatrix},\quad \mathcal T S_0 &= \begin{bmatrix}{\widecheck S_0}\\\star\end{bmatrix}, \quad C\mathcal T^{-1}= \begin{bmatrix}{\widecheck C} &
\star\end{bmatrix}
  \end{aligned}
  \end{equation*}
with $\widecheck A_i\in \mathbb R^{r \times r}$ ($i\in\{0, 1, \dots, m\}$), $\widecheck C\in \mathbb R^{p \times r}$ and a matrix $\widecheck S_0$ with $r$ rows. The reduced order model is now obtained by removing the dynamics of $S_2$ in \eqref{tr_sig_trans}. These variables are further set to zero, meaning that $S_2\equiv 0$ in the equations for $S_1$. This leads to the reduced model
\begin{equation}\label{rom}
  \begin{aligned}
		\dot{\widecheck S}(t)&=\widecheck  A_0 \widecheck  S(t) + \sum_{i=1}^{m} \widecheck  A_i\widecheck  S(t)u_i(t),\quad {\widecheck S}(0)={\widecheck S}_0 v,\\
		 \widecheck y_S(t)&=\widecheck {C}\widecheck S(t).
	\end{aligned}\end{equation}
This is a model like in
\eqref{rom_intro} satisfying $\widecheck y_S\approx y_S$ if the truncated Hankel singular values $\sigma_{r+1}, \dots, \sigma_n$ are small. In particular, fixing $s_0=e_1$ ($S_0=s_0=e_1$ and $v=1$), we obtain a reduced system for the truncated signature $S^N(\hat U)$. In fact, it is not possible to easily extend MOR error bounds from \cite{Redmann2023}, since stability in the bilinear system is required. In fact, an error analysis is still possible using the ideas of \cite{redmann_jamshidi}. However, this analysis is beyond the scope of this paper.

\subsection{Computation of Gramians}

We discuss the computation of the Gramians $P$ and $Q$, required in the MOR procedure in Subsection \ref{sec_rom}, for the particular case when $S(t)=S_{0, t}^N(\hat U)$ in \eqref{tr_sig_ode}.
This means that the matrices $A_0, A_1, \dots, A_m$ entering \eqref{tr_sig_ode} are generated according to Proposition \ref{SigODE}.
\begin{remark}
  \label{rem:nilpotency-matrices}
  Let the matrices $A_0, A_1, \dots, A_m$ of \eqref{tr_sig_ode} be like in  Proposition \ref{SigODE}. Then, they are nilpotent with order $N+1$. This means that for any $i_1, \ldots, i_N, i_{N+1} \in \{0, 1, \ldots, m\}$ and $1 \leq k \leq N$, we have
  \begin{equation*}
    A_{i_1} \cdot A_{i_2} \cdots A_{i_k} \neq 0, \text{ but } A_{i_1} \cdot A_{i_2} \cdots A_{i_N} \cdot A_{i_{N+1}}=0.
  \end{equation*}
Using this property of $A_0, A_1, \dots, A_m$, we further obtain that $K$ defined in \eqref{def_kron_mat} is nilpotent with order $2 N +1$. This means that $K^{2N +1}= 0$ while $K^{i}\neq 0$ for $i\in\{1, \dots, 2N\}$.
\end{remark}
Given the situation of  Remark \ref{rem:nilpotency-matrices}, we know from Subsection \ref{sec_gram} that \begin{align*}
 \vect(P) &= \int_0^T \vect\big(Z(t, S_0 S_0^\top)\big) \mathrm{d}t = \int_0^T\expn^{K t} \vect(S_0 S_0^\top) \mathrm{d}t =
  \sum_{j=0}^{2N} \int_0^T \frac{t^{j}}{j!}  \mathrm{d}t\,  K^j\vect(S_0 S_0^\top)\\
  &= \sum_{j=0}^{2N} \frac{T^{j+1}}{(j+1)!} K^j\vect(S_0 S_0^\top)
 \end{align*}
as well as  \begin{align*}
 \vect(Q) = \int_0^T \vect\big(Z^*(t, C^\top C)\big) \mathrm{d}t = \int_0^T\expn^{K^\top t} \vect(C^\top C) \mathrm{d}t =
  \sum_{j=0}^{2N} \frac{T^{j+1}}{(j+1)!} (K^\top)^j\vect(C^\top C).
 \end{align*}
Generally, it is not a good idea to compute such Gramians from a vectorized form, since we have to deal with an $n^2\times n^2$ matrix $K$, where already $n$ is large. However, this special situation of having extremely sparse matrices $A_0, A_1, \dots, A_m$ like in  Proposition \ref{SigODE} allows to derive $P$ and $Q$ from vectorizations in moderate high dimensions. There seems to be no alternative at the moment as we are in the time-limited case, where the Gramians are not easily accessible via matrix equations. To illustrate this, let us integrate \eqref{eqZ} and \eqref{eqZad} over the interval $[0, T]$. The resulting matrix equations for $P$ and $Q$ involve unknown terms $Z(T, S_0 S_0^\top)$ and  $Z^*(T, C^\top C)$ that need to be precomputed. This is very challenging as well. In some situations it might be beneficial to not work with $K$, but the associated Lyapunov operators in the context of computing $P$ and $Q$. Therefore, let us introduce an operator $\mathcal L$ and its adjoint via\begin{align*}
\mathcal L(X)= A_0 X+X A_0^\top + \sum_{i=1}^m A_i X A_i^\top,\quad
\mathcal L^*(X)= A_0^\top X+X A_0 + \sum_{i=1}^m A_i^\top X A_i,
                                                       \end{align*}
 where $X$ is a matrix of suitable dimensions. We can devectorize the above identities for $\vect(P)$ and $\vect(Q)$ leading to
 \begin{align}\label{rep_P_Q}
P = \sum_{j=0}^{2N} \frac{T^{j+1}}{(j+1)!} \mathcal L^j(S_0S_0^\top), \quad
Q  =\sum_{j=0}^{2N} \frac{T^{j+1}}{(j+1)!} (\mathcal L^*)^j(C^\top C).
 \end{align}
In \eqref{rep_P_Q}, the expression $\mathcal L^j$ stands for $j$-times applying the operator $\mathcal L$.

	\section{Numerical experiments}

\subsection{Learning the matrix $C$}\label{sec_learn_C}

Corollary \ref{cor_approx_y} ensures that the output $y$ in \eqref{eq_y} can be approximated arbitrary well by a linear map of the truncated signature function $S(t; u):=S^N_{0,t}(\hat{U})$, $t\in [0, T]$, given that the degree $N$ is sufficiently large and that the functions $f_0$, $f$, $c$ are at least locally Lipschitz with $f_0$, $f$ satisfying a one-sided linear growth condition. The unknown linear map is characterized by a matrix $C$ that we need to compute in practice. We fix a set of training controls $\{\bar u^{(1)}, \dots, \bar u^{(\mathfrak n_u)}\}$ within the learning procedure and determine $C$ from the following least squares problem \begin{align}\label{least_sqaures}
 \min_{C\in\mathbb R^{p\times n}}\sum_{k=1}^{\mathfrak n_u} \sum_{i=1}^{\mathfrak n} \|y(t_i; \bar u^{(k)})-C S(t_i; \bar u^{(k)}) \|^2,                                                                                                                                                                                                                                                                                                                                                                                                                                                                                                                                                                                                                                                                \end{align}
where $0=t_1<t_2<\dots <t_{\mathfrak n}=T$ is a partition of $[0, T]$. The main challenge here is to collect linear independent data $y(t_i; \bar u^{(k)})$ to have a well-posed regression problem in \eqref{least_sqaures}. This heavily depends on finding suitable training controls that also guarantee a good approximation by the truncated signature for a large set of test controls. If the data does not contain enough information, it is possible to regularize the problem, e.g., in the sense of Tikhonov (ridge regression). In this paper, we usually assume the controls to be in $L^2([0, T]; \mathbb R^m)$ which is also true for the training set. However, the regression procedure to solve \eqref{least_sqaures} highly benefits from randomized irregular control which we use in the numerical experiments below. In fact, we choose \begin{align*}
\bar u^{(k)}(t)=c_w\frac{\mathrm{d}}{\mathrm{d}t} W(t, \omega_k),
\end{align*}
where $W$ is an $m$-dimensional standard Wiener process, $c_w>0$ is a constant and $\omega_k\in\Omega$ is an outcome. Consequently, $\bar u^{(k)}$ is a realization of white noise meaning that we computed $C$ in the associated Ito SDE setting, see Remark \ref{rem_link_sde}. We observe that this leads to a well-posed problem \eqref{least_sqaures} for the numerical examples that we study. Additionally, the approximation quality is high for all test controls $u\in L^2([0, T]; \mathbb R^m)$ used in our simulations.
\begin{remark}
The step of learning the matrix $C$ is crucial for our approach. This procedure is generally very challenging. On the other hand, $C$ is only data-informed, because we are solving \eqref{least_sqaures} requiring only output data of the form $y(t_i; \bar u^{(k)})$. Therefore, we do not need any information on the original nonlinear differential equation \eqref{ODE} or its output in \eqref{eq_y}. This makes the signature approach interesting in the context of data-fitting by (low-order) bilinear systems.
\end{remark}
Below, we illustrate that the signature is also meaningful when the complexity of large-scale nonlinear systems is supposed to be reduced.

\subsection{Approximating high-dimensional nonlinear systems}

We formally consider the following controlled reaction diffusion equation
\begin{equation}\label{control_reaction_diff}
\begin{aligned}
 \frac{\partial}{\partial t} v_t(\zeta) =\frac{\partial^2}{\partial \zeta^2} v_t(\zeta) &+ \mathfrak{f}_0\big(v_t(\zeta)\big)  + \sum_{i=1}^m \mathfrak f_i\big(\zeta, v_{t}(\zeta)\big) u_i(t),\quad \zeta\in (0, 1), \quad t\in (0, T), \\
 v_0(\cdot)&\equiv 0.5 \sin(\cdot), \quad
  v_t(0) = 0\quad\text{and}\quad  v_t(1) = 0,                                                                                                                                 \end{aligned}
\end{equation}
with  Dirichlet boundaries. We apply a finite difference scheme to \eqref{control_reaction_diff}. We fix a spatial step size $h:=\frac{1}{d+1}$ and introduce a grid by $\zeta_j = j h$ for $j=0, 1, \dots, d+1$. We obtain \begin{equation}    \label{finite_difference_model}
\begin{aligned}
\dot x_1(t) &= \frac{x_2(t)-2x_1(t)}{h^2} +\mathfrak{f}_0\big(x_1(t)\big) + \sum_{i=1}^m \mathfrak f_i\big(\zeta_1, x_1(t)\big) u_i(t),\\
\dot x_j(t) &= \frac{x_{j+1}(t)-2x_j(t)+x_{j-1}(t)}{h^2} +\mathfrak{f}_0\big(x_j(t)\big) + \sum_{i=1}^m \mathfrak f_i\big(\zeta_j, x_j(t)\big) u_i(t), \\
\dot x_d(t) &= \frac{-2x_d(t)+x_{d-1}(t)}{h^2} + \mathfrak{f}_0\big(x_d(t)\big) + \sum_{i=1}^m \mathfrak f_i\big(\zeta_d, x_d(t)\big) u_i(t)                                                                                           \end{aligned}
\end{equation}
for $j\in\{2, \dots, d-1\}$. The intuition is that $x_i(t)\approx v_t(\zeta_i)$ for $i\in\{1, \dots, d\}$. We choose $m=2$, $T=1$ and $d=1000$ as well as $\mathfrak{f}_0(v)=(v-v^3)$, $\mathfrak f_1\big(\zeta, v\big)=\expn^{\zeta}$ and $\mathfrak f_2\big(\zeta, v\big)=v^2$ for $v\in \mathbb R$. This leads to $f_0(x)= \mathcal Ax + x - x^{\circ 3}$, $f_1(x)=B=\big(\expn^{\zeta_1}, \dots, \expn^{\zeta_d}\big)^\top$ and $f_2(x)=x^{\circ 2}$ in \eqref{ODE}, where $\mathcal A\in\mathbb R^{d\times d}$ is a discrete version of the second derivative, $f=(f_1, f_2)$ and $x\in\mathbb R^d$. We introduce the quantity of interest as follows \begin{align}\label{output_num}
		y(t)=\exp\bigg\{\frac{1}{d}\sum_{j=1}^d x_j(t)\bigg\}
	\end{align}
meaning that $c(x)=\exp\big\{\frac{1}{d}\sum_{j=1}^d x_j\big\}$ in \eqref{eq_y}. Consequently, we have an example, in which $f_0$, $f$ and $c$ are not globally Lipschitz. We use test controls for the experiments that are easy to parameterize. These are \begin{align}\label{test_control}
  u^{(k)}(t)=\left(\begin{matrix} u_1^{(k)}(t)\\u_2^{(k)}(t)\end{matrix}\right)=\left(\begin{matrix} \cos(k t)\\ \sin(kt)\end{matrix}\right),\quad k=1,\dots, 1000.\end{align}
\begin{figure}[ht]
\center
\includegraphics[width=8cm,height=7cm]{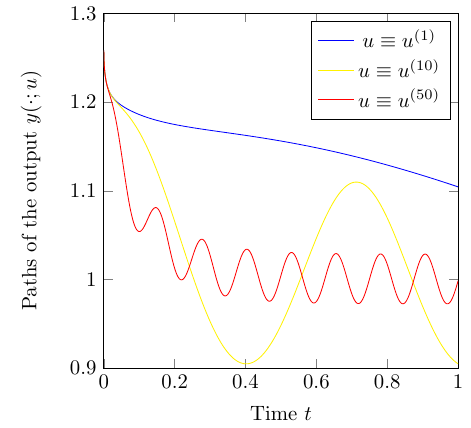}
\caption{Three example for $y(\cdot, u)$ in \eqref{output_num} with $u= u^{(1)}, u^{(10)}, u^{(50)}$.}\label{fig1}
\end{figure}
We show examples of the output in \eqref{output_num} using three different test controls from \eqref{test_control}, see Figure \ref{fig1}. Fixing $k=1, 10, 50$ in this figure, we can see that functions that we try to reproduce by the signature model differ significantly. The first step now is to approximate $y$ in \eqref{output_num} by $y_S=CS(\cdot; u)$, where $S(t; u):=S^N_{0,t}(\hat{U})$ is the truncated signature of order $N$. This means that system \eqref{tr_sig_ode} is considered with $s_0=e_1$ and matrices $A_0, A_1, \dots, A_m$ like in Proposition \ref{SigODE}. We compute $C$ according to the procedure explained in Section \ref{sec_learn_C}. We simulate $\mathfrak n_u=1000$ paths of a standard Wiener process $W$ and fix $c_w=0.2$ to determine $C$ from the least squares problem \eqref{least_sqaures}. We check for the quality of the resulting matrix $C$ by computing the following error \begin{align}\label{def_l2_dis}
\mathcal E_{\sig}:=\|y-y_S \|_{L^2}:=\sqrt{\frac{1}{1000} \sum_{k=1}^{1000} \int_0^T \|y(t; u^{(k)})- y_S(t; u^{(k)}) \|^2\mathrm{d}t}.
\end{align}
We intend to have a (bi)linearization error $\mathcal E_{\sig}$ smaller than one percent. This can be achieved by setting $N=5$ resulting in $\mathcal E_{\sig}=3.7302$e$-03$. The signature approximation error is depicted by a red solid line in Figure \ref{fig2}. Having $N=5$ leads to a truncated signature dimension of $n=364$. This already provides a significant complexity reduction as we now deal with a bilinear system taking into account that $S(\cdot; u)$ solves the bilinear system \eqref{tr_sig_ode} with output $y_S(\cdot; u) = C S(\cdot; u)$. Moreover, this bilinear system has a lower dimension than the original problem.  We optimize the dimension of the $n=364$-dimensional system \eqref{tr_sig_ode} further by applying the MOR procedure explained in Section \ref{sec_mor}. Analogue to \eqref{def_l2_dis} we define the signature dimension reduction error $\mathcal E_{\mor}:= \|y_S - \widecheck y_{S} \|_{L^2}$,  where $\widecheck y_{S}$ is the  output of the reduced model \eqref{rom}, and illustrate it in Figure \ref{fig2} as well.  The total error involving both the bilinearization and the MOR error is $\mathcal E_{\reduced-\sig}:= \|y - \widecheck y_{S} \|_{L^2}$. $\mathcal E_{\reduced-\sig}$ now depends on the dominant error of both approximations.
\begin{figure}[ht]
 \begin{minipage}{0.49\linewidth}
  \hspace{-0.5cm}
 \includegraphics[width=1.0\textwidth,height=6cm]{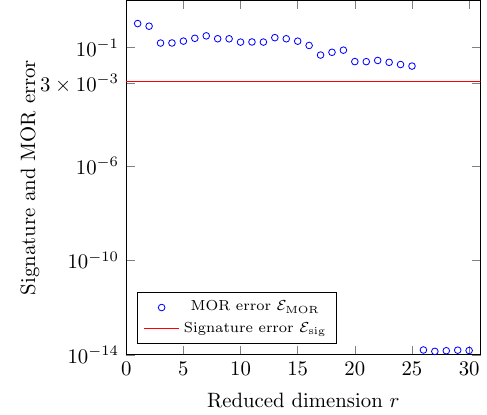}
 \caption{Error signature approximation of $y$ in \eqref{output_num} (solid line) and error when reducing the dimension of the signature (circles).}\label{fig2}
 \end{minipage}
 \hspace{0.2cm}
 \begin{minipage}{0.49\linewidth}
 \includegraphics[width=1.0\textwidth,height=6cm]{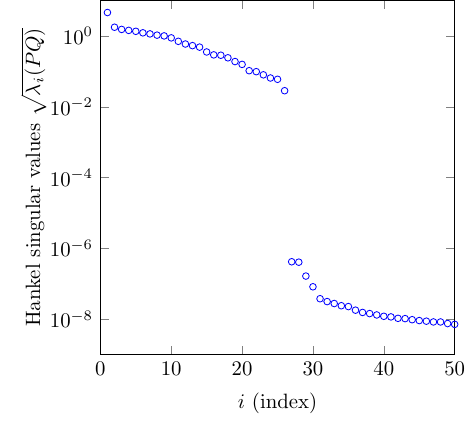}
 \caption{First $50$ Hankel singular singular values of the signature model approximating the output in  \eqref{output_num}.}\label{fig3}
 \end{minipage}
 \end{figure}
We observe from Figure \ref{fig2} that the MOR error $\mathcal E_{\mor}$ is very large for $r\leq 16$ since this error is above ten percent. Reduced order models with dimension $20\leq r\leq 25$ perform with an error of around two percent but in such a scenario the MOR approach will still dominate the entire approximation, i.e., $\mathcal E_{\reduced-\sig}=\mathcal E_{\mor}$. Suddenly, $r\geq 26$ basically leads to an exact reduced model when estimating the signature process. Therefore, we obtain $\mathcal E_{\reduced-\sig}=\mathcal E_{\sig}$. The behavior of the MOR error follows the decay of the Hankel singular values introduced in Proposition \ref{prop_bal}. These are depicted in Figure \ref{fig3}. This emphasizes that the Hankel singular values deliver an excellent a-priori criterion for the expected MOR error in our approach. In summary, we can say that we found a bilinear system \eqref{rom} of dimension $r=26$ that well-approximates a $d=1000$-dimensional system \eqref{finite_difference_model} with polynomial nonlinearities. This illustrates how powerful our approach can be.

\appendix

\section{Proofs of Section \ref{3_1}}

We need the following lemma of Gronwall for the proofs below.

\begin{lem}[Gronwall's lemma -- integral form]\label{gron_int} Suppose that $f, \alpha: [0, T] \rightarrow \mathbb R$ are measurable bounded functions and $\beta: [0, T]\rightarrow \mathbb R$ is a nonnegative integrable function.
If \begin{align*}
    f(t)\leq \alpha(t)+\int_{0}^t \beta(s) f(s) ds
   \end{align*}
for all $t\in [0, T]$, then it holds that \begin{align}\label{gronwallineq}
    f(t)\leq \alpha(t)+\int_{0}^t \alpha(s)\beta(s) \exp\left\{\int_s^t \beta(v)dv\right\} ds
   \end{align}
for all $t\in [0, T]$. A non-decreasing and continuous function $\alpha$ implies that \eqref{gronwallineq} becomes \begin{align}\label{gronwallineq2}
    f(t)\leq \alpha(t)\exp\left\{\int_{0}^t \beta(s)ds\right\}
   \end{align}
   for all $t\in [0, T]$.
\end{lem}
\begin{proof}
There is a proof of this statement in \cite[Lemma A.3]{redmann2025}.
\end{proof}

\subsection{Proof of Lemma \ref{Hilfslemma}}\label{sec_proof_Lemma0}

\begin{proof}[Proof of Lemma \ref{Hilfslemma}]
Using the triangle inequality and the assumptions in (a), we obtain
	\begin{align*}
	 \|x(t)\|&\leq \|x_0\|+\int_{0}^{t}\|f_0(x(s))\| \, \mathrm{d}s+\int_{0}^{t}\|f(x(s))\|\|u(s)\| \, \mathrm{d}s\\
	 &\leq \|x_0\|+\int_{0}^{t}\mathcal C(1+\|x(s)\|) \, \mathrm{d}s+\int_{0}^{t}\mathcal C(1+\|x(s)\|)\|u(s)\| \, \mathrm{d}s\\
	 &\leq \|x_0\|+\mathcal C(T+\|u\|_{L^1([0,T])})+\int_{0}^{t}\mathcal C(1+\|u(s)\|)\|x(s)\|\, \mathrm{d}s
	\end{align*}
for all $t\in [0, T]$. Applying \eqref{gronwallineq2} yields
\begin{align}\nonumber
   \|x(t)\|&\leq (\|x_0\|+\mathcal C(T+\|u\|_{L^1([0,T])}))\exp\left\{\int_{0}^t \mathcal C(1+\|u(s)\|)ds\right\}\\ \label{x_in_ball}
   &\leq (\|x_0\|+\mathcal C(T+R))\exp\left\{\mathcal C(T+R)\right\}
   \end{align}
   for all $t\in [0, T]$. On the other hand,
    the product rule and \eqref{ODE} provide
   \begin{align}\nonumber
\|x(t)\|^2&=x(t)^\top x(t)=\|x_0\|^2+2 \int_{0}^{t} x(s)^\top \, \mathrm{d}x(s)\\ \label{apply_prod_rule}
&=\|x_0\|^2+2 \int_{0}^{t} x(s)^\top  f_0(x(s)) \, \mathrm{d}s+2 \int_{0}^{t} x(s)^\top  f(x(s))u(s)\, \mathrm{d}s.
			\end{align}
We insert the following estimate
\begin{align}\nonumber
   2x(s)^\top  f(x(s))u(s) &
   =2\trace\Big(f(x(s))u(s) x(s)^\top\Big)
   \leq \big\|f(x(s))\big\|^2_F + \big\|u(s) x(s)^\top\big\|^2_F\\ \label{est_square}
   &=\big\|f(x(s))\big\|^2_F + \|x(s)\|^2\|u(s)\|^2
                 \end{align}
into \eqref{apply_prod_rule} and exploit the one-sided linear growth condition in (b) leading to \begin{align*}
\|x(t)\|^2&\leq \|x_0\|^2+ \int_{0}^{t} 2\langle x(s), f_0(x(s))\rangle + \big\|f(x(s))\big\|^2_F \, \mathrm{d}s+\int_{0}^{t} \|x(s)\|^2\|u(s)\|^2\, \mathrm{d}s\\
&\leq \|x_0\|^2+ \int_{0}^{t} \mathcal C(1+\|x(s)\|^2) \, \mathrm{d}s+\int_{0}^{t} \|x(s)\|^2\|u(s)\|^2\, \mathrm{d}s\\
&\leq \|x_0\|^2+\mathcal CT+ \int_{0}^{t} (\mathcal C+\|u(s)\|^2)\|x(s)\|^2 \, \mathrm{d}s
			\end{align*}
for all $t\in [0, T]$. From \eqref{gronwallineq2} we obtain \begin{align}\label{x_in_ball2}
\|x(t)\|^2\leq (\|x_0\|^2+\mathcal CT)\exp\left\{\mathcal CT+\|u\|_{L^2([0,T])}^2\right\}\leq (\|x_0\|^2+\mathcal CT)\exp\left\{\mathcal CT+R^2\right\}
\end{align}
for all $t\in[0, T]$.  Hence, in both (a) and (b), $x$ takes values in a ball with radius being independent of the particular control. Since $f$ is continuous, it is bounded by a constant independent of $u$ on this ball. This concludes the proof.

% 	\begin{align*}
% 			\|f(x(t))\| \leq C(1+\|x(t)\|)\leq  C(1+\|x_0\|+\int_{0}^{t}\|f(x(s))\|\|u(s)\| \, \mathrm{d}s)
% 	\end{align*}
% 	an application of Gronwall yields
% 	\begin{align*}
% 	\|f(x(t))\| \leq (C+C\|x_0\|)\exp(C\|u\|_{L^1([0,T])})\leq (C+C\|x_0\|)\exp(C\cdot R)
% 	\end{align*}
	\end{proof}

\subsection{Proof of Theorem \ref{Lemma1}}\label{sec_proof_Lemma1}

\begin{proof}[Proof of Theorem \ref{Lemma1}]
Let us write $x(t):=x(t; u)$ and $z(t):=x(t; v)$ for simplicity of the notation.
Adding a zero and exploiting the triangle inequality, we obtain
		\begin{align*}
			&\|x(t)-z(t)\|\\&=\|\int_{0}^t f_0(x(s))-f_0(z(s)) \, \mathrm{d}s+\int_{0}^tf(x(s))\Big(u(s)-v(s)\Big)+\Big(f(x(s))-f(z(s))\Big) v(s) \mathrm{d}s \|\\
			&\leq \int_{0}^{t} \|f_0(x(s))-f_0(z(s))\|\,\mathrm{d}s +\int_{0}^t \|f(x(s))\| \|u(s)-v(s)\|+\|f(x(s))-f(z(s))\|\|v(s)\| \, \mathrm{d}s.\end{align*}
As the global Lipschitz property implies the global linear growth condition, we can apply Lemma \ref{Hilfslemma} case (a). Hence, by the Lipschitz continuity of $f_0$ and $f$, we have
			\begin{align*}
\|x(t)-z(t)\|&\leq \int_{0}^t  (L+L\|v(s)\|)\|x(s)-z(s)\| \, \mathrm{d}s + K_R \|u-v\|_{L^1([0,T ])}.
			\end{align*}
Gronwall's lemma in \eqref{gronwallineq2} yields
			\begin{align*}
				\|x(t)-z(t)\|\leq  \exp\{ L T+L R\}\cdot K_R\cdot\|u-v\|_{L^1([0,T ])}.
			\end{align*}
Therefore, the result of this theorem follows.
	\end{proof}

\subsection{Proof of Theorem \ref{Lemma2}}\label{sec_proof_Lemma2}

\begin{proof}[Proof of Theorem \ref{Lemma2}]
 Let us use $x(t)$ and $z(t)$ instead of $x(t; u)$ and $x(t; v)$, respectively. We apply the product rule and \eqref{ODE} resulting in
		\begin{align*}
		&	\|x(t)-z(t)\|^2=(x(t)-z(t))^\top(x(t)-z(t))=2 \int_{0}^{t} (x(s)-z(s))^\top \, \mathrm{d}(x(s)-z(s))\\
			&=2 \int_{0}^{t} (x(s)-z(s))^\top  (f_0(x(s))-f_0(z(s))) \, \mathrm{d}s+2 \int_{0}^{t} (x(s)-z(s))^\top  (f(x(s))u(s)-f(z(s))v(s)) \, \mathrm{d}s\\
			&=2 \int_{0}^{t} (x(s)-z(s))^\top  (f_0(x(s))-f_0(z(s))) \, \mathrm{d}s\\
			&\quad+2 \int_{0}^{t} (x(s)-z(s))^\top  (f(x(s))\Big(u(s)-v(s)\Big)+\Big(f(x(s))-f(z(s))\Big)v(s)) \, \mathrm{d}s.
		\end{align*}
We exploit the following estimates to the last line of the above identities: \begin{align*}
&2\langle x(s)-z(s),  f(x(s))\big(u(s)-v(s)\big)\rangle \leq \|x(s)-z(s)\|^2 + \|f(x(s))\|^2\|u(s)-v(s)\|^2, \\
&2\langle x(s)-z(s),  \Big(f(x(s))-f(z(s))\Big)v(s)\rangle \leq \|f(x(s))-f(z(s))\|_F^2+\|x(s)-z(s)\|^2\|v(s)\|^2,
\end{align*}
where the last inequality is derived analogously to \eqref{est_square}. This results in \begin{align*}
\|x(t)-z(t)\|^2
&\leq  \int_{0}^{t} 2\langle x(s)-z(s), f_0(x(s))-f_0(z(s))\rangle +\|f(x(s))-f(z(s))\|_F^2 \, \mathrm{d}s\\
&\quad+ \int_{0}^{t} \|x(s)-z(s)\|^2(1+\|v(s)\|^2) \, \mathrm{d}s+ \int_{0}^{t} \|f(x(s))\|^2\|u(s)-v(s)\|^2 \, \mathrm{d}s.
		\end{align*}
We apply the one-sided Lipschitz property and Lemma \ref{Hilfslemma} case (b) giving us  \begin{align*}
\|x(t)-z(t)\|^2\leq  \int_{0}^{t} \|x(s)-z(s)\|^2(L+1+\|v(s)\|^2) \, \mathrm{d}s+ K_R^2\|u-v\|^2_{L^2([0, T])}
\end{align*}
for all $t\in[0, T]$. Without loss of generality we can assume that $L\geq 0$, so that Lemma \ref{gron_int} can be applied. This leads to
\begin{align*}
\|x(t)-z(t)\|^2\leq  \exp\left\{T(L+1)+R^2\right\} K_R^2\|u-v\|^2_{L^2([0, T])}
\end{align*}
for all $t\in [0, T]$ and hence concludes the proof.
	\end{proof}

\section*{Acknowledgments}
 MR and JW are supported by the DFG via the individual grant ``Low-order approximations for large-scale problems arising in the context of high-dimensional
PDEs and spatially discretized SPDEs''-- project number 499366908.
	\bibliography{bibfile}
	\bibliographystyle{plain}
%	\printbibliography
	
\end{document}